\documentclass{amsart}
\usepackage[british]{babel}
\usepackage{amssymb,stmaryrd,enumerate,hyperref,xcolor,latexsym}

\newtheorem{theorem}{Theorem}[section]
\newtheorem{corollary}[theorem]{Corollary}
\newtheorem{definition}[theorem]{Definition}
\newtheorem{proposition}[theorem]{Proposition}

{\theoremstyle{remark}\newtheorem{example}[theorem]{Example}}
\newtheorem{lemma}[theorem]{Lemma}

\newcommand{\Kappa}{\mathrm{K}}
\newcommand{\assign}{:=}
\newcommand{\backassign}{=:}
\newcommand{\dottimes}{\mathbin{\dot{\times}}}

\newcommand{\of}{:}
\newcommand{\suchthat}{:}
\newcommand{\tmem}[1]{{\em #1\/}}
\newcommand{\tmop}[1]{\ensuremath{\operatorname{#1}}}
\newcommand{\tmstrong}[1]{\textbf{#1}}

\newcommand{\tmtextit}[1]{\text{{\itshape{#1}}}}

\newcommand{\tmtextup}[1]{\text{{\upshape{#1}}}}
\newcommand{\tmxspace}{\hspace{1em}}

\newenvironment{enumeratealpha}{\begin{enumerate}[a{\textup{)}}] }{\end{enumerate}}
\newenvironment{itemizedot}{\begin{itemize} }{\end{itemize}}

\hypersetup{
    colorlinks = true,
    linkcolor = {blue},
    citecolor = {magenta},
    urlcolor = {black}
}

\begin{document}

\title{Sign sequences of log-atomic numbers}

\author{Vincent Bagayoko}
\address{IMJ-PRG (Paris)}
\email{bagayoko@imj-prg.fr}

\date{February 24, 2024}

\begin{abstract}
  Log-atomic numbers are surreal numbers whose iterated logarithms are
  monomials, and consequently have a trivial expansion as transseries. Presenting surreal numbers
  as sign sequences, we give the sign sequence formula for log-atomic numbers.
  To that efect, we relate log-atomic numbers to fixed-points of certain
  surreal functions.
\end{abstract}

{\maketitle}

\section*{Introduction}

Each surreal number is precisely described by a sequence of signs giving the
path to that number in the surreal tree. Yet the relation between this sign sequence and properties of numbers in the ordered, differential exponential field of surreal numbers is still complex and mysterious. Originally, Conway's class
$\mathbf{No}$\label{autolab1} of {\tmem{surreal numbers}} of Conway
{\cite{Con76}} is an inductively defined ordered field with additional
structure. Any number $x \in \mathbf{No}$ is constructed from sets $L, R$ of
previously defined numbers, as the ``simplest'' number $x = \{ L|R \}$ filling the cut $(L,
R)$, i.e. with $L < x$ and $x < R$.
Thus $\mathbf{No}$ comes equipped
with an inductively defined order $<$ and a corresponding ordinal rank called the birthday $\beta (x)$ which
represents the minimal ordinal number of inductive steps required to construct
$x$. For instance $0 \assign \{ \varnothing | \varnothing \}$ has birthday $0$
whereas $1 \assign \{ \{ 0 \} | \varnothing \}$ has birthday $1$ and $1 / 2
\assign \{ \{ 0 \} | \{ 1 \} \}$ has birthday $2$.

Conway defined the arithmetic on surreal numbers,
turning $(\mathbf{No}, +, \cdot, <)$ into an ordered field extension of the
reals, naturally containing the class $\mathbf{On}$\label{autolab2} of ordinal numbers. He also discovered that $\mathbf{No}$ enjoys a natural structure of
field of Hahn series as per~{\cite{Hahn1907}}. Every surreal number can be
expressed as a possibly transfinite sum of additively irreducible numbers
called {\tmem{monomials}}. Subsequently, the ordered field structure on surreal numbers was enriched with an exponential function \cite{Con76,Gon86} and a derivation \cite{BM18}, culminating in the presentation of $\mathbf{No}$ as an elementary extension of all maximal Hardy fields \cite{ADH:H-closed}.

The sign sequence presentation of surreal numbers, studied by Gonshor
{\cite{Gon86}}, is an alternative way to define surreal numbers. In
this picture, numbers are sequences of signs $+ 1, - 1$ indexed by ordinals,
or equivalently, nodes in the binary tree~$\{ - 1, + 1 \}^{< \mathbf{On}}$.
The birthdate coincides with the domain, called {\tmem{length}}, of the
sign sequences. A partial well-founded ordering $\sqsubset$\label{autolab3} of {\tmem{simplicity}}, which corresponds
to the inclusion of sign sequences in one another, emerges as a more precise measure of the
complexity of numbers.

It is difficult in general to compute sign sequences. For instance, given numbers $x, y$ with known
sign sequences, computing the sign sequence of $x + y$, or even formulating
the properties of that of $xy$ in general are open problems \cite{Gon86}. However, sign sequences behave well with respect to operations that preserve
simplicity, and they allow for accurate descriptions of
properties of such functions.

A subclass of $\mathbf{No}$ of particular interest is the class $\mathbf{La}$ of
{\tmem{$\log$-atomic numbers}}. Those are numbers $\mathfrak{a}$ such that the
$n$-fold iteration $\log_n \mathfrak{a}$ of the logarithm at $\mathfrak{a}$
yields a monomial for each $n \in \mathbb{N}$. The class $\mathbf{La}$ was characterised by Berarducci and Mantova
{\cite{BM18}} in order to define a derivation on $\mathbf{No}$ that
is compatible with the exponential and the structure of field of series. It plays a fundamental role in the investigation of the properties of expansions of
numbers as transseries \cite{BvdH19,BM18,Schm01}, as log-atomic numbers are those whose transserial
expansion is trivial. Finally $\mathbf{La}$ is a key element in the
definition {\cite{BvdHM:surhyp}} of the first surreal hyperexponential
function $E$, a very fast and regularly growing "analytic" function satisfying Padgett's first-order axioms \cite{Pad:phd} for Kneser's transexponential function \cite{Kn49}. Our our main goal in this article is to compute the sign
sequence formula for the unique strictly increasing bijective parametrisation $\Xi_{\mathbf{La}} \colon \mathbf{No} \longrightarrow \mathbf{La}$ that preserves simplicity. In this, we are continuing work of Kuhlmann and Matusinski
{\cite{KM15}}, who considered a proper subclass $\mathbf{K}$
of $\mathbf{La}$ and gave the sign sequence formula for its parametrisation. We will rely on their sign sequence formula to derive ours.

We introduce preliminary notions on
surreal numbers and their arithmetic in
Section~\ref{section-numbers-sign-sequences}
In Section~\ref{section-surreal-substructures}, we state facts from {\cite{BvdH19}} about of surreal substructure. Those are subclasses of $\mathbf{No}$ which have a unique simplicity preserving and strictly increasing parametrisation by $\mathbf{No}$. They include monomials and $\log$-atomic numbers.
Section~\ref{section-exponentiation} gives a presentation of surreal
exponentiation that is tailored to our goals. We show (Proposition~\ref{prop-E-lambda}) that $E$ conincides with $\Xi_{\mathbf{La}}$ on the elementary substructure of $(\mathbf{No},+,\cdot,\exp)$ of numbers $a >\mathbb{R}$ with length strictly below the first
$\varepsilon$-number $\varepsilon_0 \in \mathbf{On}$, thus giving a method for computing $E$ on this set. 
In Section~\ref{section-log-atomic-fixed-points}, we prove our main technical result Theorem~\ref{th-log-atomic-caracterization} which characterises
$\log$-atomic numbers in terms of fixed points of surreal substructures. 
We ths result in Section~\ref{section-applications}. We show (Proposition~\ref{prop-La-closed}) that the $\mathbf{La}$ is closed under taking
suprema in $(\mathbf{No},\sqsubset)$, which is a requirement for the possibility of describing a sign sequence formula.  We then give the sign sequence formula (Theorem~\ref{th-formula}).

\section{Numbers and sign sequences}\label{section-numbers-sign-sequences}

\subsection{Numbers as sign sequences}

\begin{definition}
  A surreal number is a map $x \of \ell (x) \rightarrow \{ - 1, 1 \} ; \alpha
  \mapsto x [\alpha]$\label{autolab4}, where $\ell (x) \in
  \mathbf{On}$\label{autolab5} is an ordinal number. We call $\ell (x)$ the
  \tmtextup{{\tmstrong{length}}} of $x$.
\end{definition}

Given a surreal number $x \in \mathbf{No}$, we extend its sign sequence with
$x [\alpha] = 0$ for all $\alpha \geqslant \ell (x)$. Given $x \in
\mathbf{No}$ and $\alpha \in \mathbf{On}$, we also introduce the
{\tmem{restriction}} $y = {x \upharpoonleft \alpha} \in \mathbf{No}$ to
$\alpha$ as being the initial segment of $x$ of length $\alpha$, i.e. $y
[\beta] = x [\beta]$ for $\beta < \alpha$ and $y [\beta] = 0$ for~$\beta
\geqslant \alpha$.

The ordering $<$ on $\mathbf{No}$ is lexicographical: given distinct elements
$x, y \in \mathbf{No}$, there exists a smallest ordinal $\alpha$ with $x
[\alpha] \neq y [\alpha]$, and $x < y$ if and only if $x [\alpha] < y
[\alpha]$.

For $x, y \in \mathbf{No}$, we say that $x$ is
{\tmem{simpler}}{\index{simplicity}} than $y$, and write $x \sqsubseteq
y$\label{autolab6}, if $x = y \upharpoonleft \ell (x)$. We write $x \sqsubset
y$ if $x \sqsubseteq y$ and $x \neq y$. For $x \in \mathbf{No}$, we
write\label{autolab7}
\[ x_{\sqsubset} \assign \{ a \in \mathbf{No} \suchthat a \sqsubset x \} \]
for the set of numbers that are strictly simpler than~$x$. The partially
ordered class $(\textbf{} \mathbf{No}, \sqsubset)$ is well-founded, and each
set $(x_{\sqsubset}, \sqsubset)$ is well-ordered with order type $\tmop{ord}
(x_{\sqsubset}, \sqsubset) = \ell (x)$. Moreover, $x_{\sqsubset}$ is the union
of the sets
\[ x_L \assign \{ y \in \mathbf{No} \suchthat y \sqsubset x, y < x \}
   \text{\qquad and\qquad$x_R \assign \{ y \in \mathbf{No} \suchthat y
   \sqsubset x, y > x \}$.} \]
Every linearly ordered subset $C$ of $(\textbf{} \mathbf{No}, \sqsubset)$ has
a {\tmem{supremum}} $s = \sup_{\sqsubset} C$\label{autolab8} in $\textbf{}
\mathbf{No}$. Indeed, we have $\ell (s) = \sup \ell (Z)$, and $s [\alpha] = x
[\alpha]$ for all $x \in C$ with $\alpha < \ell (x)$. Numbers~$x$ that are
equal to $\sup_{\sqsubseteq} x_{\sqsubset}$ are called {\tmem{limit
numbers}}{\index{limit surreal number}}; other numbers are called
{\tmem{successor numbers}}{\index{successor surreal number}}. Limit numbers
are exactly the numbers whose length is a~limit ordinal.

Note that $\mathbf{No}$ is a proper class. In particular, the linearly ordered
class $(\mathbf{On}, \in)$ is embedded into the partial order $(\mathbf{No},
\sqsubset)$ through the map $\mathbf{On} \longrightarrow \mathbf{No} \: ; \:
\alpha \mapsto (1)_{\beta < \alpha}$. We will identify ordinals numbers with
their surreal image.

\subsection{Monomials}

For $x \in \mathbf{No}^{>}$ we define
\[ \mathcal{H} [x] \assign \{ y \in \Kappa^{>} \suchthat \exists q \in
   \mathbb{Q}^{>}, q^{- 1} x < y < qx \} . \]
The class $\mathcal{H} [x]$ has a unique $\sqsubset$-minimal element denoted
$\mathfrak{d}_x$\label{autolab9}, and the class $\mathbf{Mo} \assign \{
\mathfrak{d}_x \suchthat x \in \mathbf{No}^{>} \}$\label{autolab10} is a
subgroup of $(\mathbf{No}^{>}, \cdot)$.

\begin{definition}
  A {\tmem{{\tmstrong{monomial}}}} is a number $\mathfrak{m}$ which is
  simplest in the class $\mathcal{H} [\mathfrak{m}]$, i.e. a number of the
  form $\mathfrak{m}=\mathfrak{d}_x$ for a certain $x \in \mathbf{No}^{>}$.
\end{definition}

Conway defined {\cite[Chapter~3]{Con76}} a parametrisation $z
\mapsto \dot{\omega}^z$\label{autolab11} of $\mathbf{Mo}$. This is an
isomorphism $(\mathbf{No}, +, 0, <) \longrightarrow (\mathbf{Mo}, \cdot, 1,
<)$ which also preserves simplicity:
\[ \forall x, y \in \mathbf{No}, x \sqsubset y \Longleftrightarrow
   \dot{\omega}^x \sqsubset \dot{\omega}^y . \]
The function $\mathbf{On} \longrightarrow \mathbf{No} \: ; \: \alpha \mapsto
\dot{\omega}^{\alpha}$ coincides with the ordinal exponentiation with basis
$\omega$.

\subsection{Arithmetic on sign sequences}

For ordinals $\alpha, \beta$, we will denote their ordinal sum, product, and
exponentiation by {$\alpha \dotplus \beta$}\label{autolab12}, $\alpha
\dottimes \beta$\label{autolab13} and $\dot{\alpha}^{\beta}$\label{autolab14}.
Here, we consider the operations $\dotplus$ and $\dottimes$ of
{\cite[Section~3.2]{BvdH19}} on~$\mathbf{No}$ which extend
ordinal arithmetic to $\mathbf{No}$.

For numbers $x, y$, we write $x \dotplus y$\label{autolab15} for the number
whose sign sequence is the concatenation of that of $y$ at the end of that of
$x$. So $x \dotplus y$ is the number of length $\ell (x \dotplus y) = \ell (x)
\dotplus \ell (y)$, which satisfies
\begin{eqnarray*}
  (x \dotplus y) [\alpha] & = & x [\alpha] \hspace*{\fill} (\alpha < \ell
  (x))\\
  (x \dotplus y) [\ell (x) \dotplus \beta] & = & y [\beta] \hspace*{\fill}
  (\beta < \ell (y))
\end{eqnarray*}
The operation $\dotplus$ clearly extends ordinal sum. We have $x \dotplus 0 =
0 \dotplus x = x$ for all $x \in \mathbf{No}$.

We write $x \dottimes y$\label{autolab16} for the number of length $\ell (x)
\dottimes \ell (y)$ whose sign sequence is defined by
\begin{eqnarray*}
  (x \dottimes y) [\ell (x) \dottimes \alpha \dotplus \beta] & = & y [\alpha]
  x [\beta] \hspace*{\fill} (\alpha < \ell (y), \beta < \ell (x))
\end{eqnarray*}
The operation $\dottimes$ extends ordinal product. Given $x \in \textbf{} \mathbf{No}$ and $\alpha \in
\mathbf{On}$, the number $x \dottimes \alpha$ is the concatenation of $x$ with
itself to the right ``$\alpha$-many times'', whereas $\alpha \dottimes x$ is
the number obtained from $x$ by replacing each sign by itself $\alpha$-many
times. We have $x \dottimes 1 = x$ and $x \dottimes (- 1) = - x$ for all $x
\in \mathbf{No}$.

The operations also enjoy properties which extend that of their ordinal
counterparts as is illustrated hereafter. We refer to
{\cite[Section~3.2]{BvdH19}} for more details.

\begin{lemma}
  \label{lem-surreal-ordinal-operations}{\tmem{{\cite[Lemma~3.1]{BvdH19}}}}
THe laws $\dotplus$ and $\dottimes$ are associative with identities $0$ and $1$ respectively, and $\dottimes$ distribues with $\dotplus$ on the left.
  Moreover, for $x, y \in \textbf{} \mathbf{No}$ where $y$ is a limit, we have
  \begin{enumeratealpha}
    \item $x \dotplus y = \sup_{\sqsubset} (x
    \dotplus y_{\sqsubset})$.
    
    \item $x
    \dottimes y = \sup_{\sqsubset}  \left( x \mathbin{\dottimes}
    y_{\sqsubset} \right)$.
  \end{enumeratealpha}
\end{lemma}

\subsection{Sign sequences formulas}

If $\alpha$ is an ordinal and $f$ is a function $\alpha \longrightarrow
\mathbf{No}$, then~$(\alpha, f)$ uniquely determines a surreal number
$[\alpha, f]$ defined inductively by the rules
\begin{itemizedot}
  \item $[0, f] = 0$,
  
  \item $[\gamma \dotplus 1, f \upharpoonleft (\gamma + 1)] = [\gamma, f
  \upharpoonleft \gamma] \dotplus f (\gamma)$ for $\beta < \alpha$,
  
  \item $[\gamma, f \upharpoonleft \gamma] = \sup_{\sqsubseteq} \{ [\mu, f
  \upharpoonleft \mu] \suchthat \mu < \gamma \}$ for limit ordinals $\gamma
  \leqslant \alpha$.
\end{itemizedot}
We say that $(\alpha, f)$ is a {\tmem{sign sequence formula}}{\index{sign
sequence formula}} for $[\alpha, f]$, which we identify with the informal
expression
\[ \dot{\sum}_{\gamma < \alpha} f (\gamma) = f (0) \dotplus f (1) \dotplus
   \mathord{\cdots \dotplus f (\beta) \dotplus \mathord{\cdots}} \]
where $\gamma$ ranges in $\alpha$. For instance, the ordered pair $(\omega,
((- 1)^n)_{n < \omega})$ is a sign sequence formula of $1 / 2 \dottimes \omega
= 1 \dotplus (- 1) \dotplus 1 \dotplus (- 1) \dotplus \cdots = 2 / 3$. In
general, we look for such formulas where $f$ alternates between
$\mathbf{On}^{>}$ and $- \mathbf{On}^{>}$, that is, where $f$ ranges in
$\mathbf{On}^{>} \cup - \mathbf{On}^{>}$ and where, for each ordinal $\beta$
with $\beta \dotplus 1 < \alpha$, we have $f (\beta) f (\beta \dotplus 1) <
0$. It is easy to see that every surreal number admits a unique such
alternating sign sequence formula. We refer to this formula as {\tmem{the}}
sign sequence formula of said number.

If $F \of \textbf{} \mathbf{No} \longrightarrow \textbf{} \mathbf{No}$ is a
function, we may look for a function $\Psi$ whose value at each number $x$ is
a sign sequence formula of $F (x)$. We then consider $\Psi$ as a {\tmem{sign
sequence formula}} for $F$.

As an example, we now state Gonshor's results regarding the sign sequences of
monomials. For $z \in \mathbf{No}$, we write $\tau_z$ for the order type of
$(z_L, \sqsubset)$, that is the order type $z^{- 1} [\{ 1 \}]$ when $z$ is
seen as a sign sequence, or equivalently the ordinal number of signs $1$ in
the sign sequence of~$z$. We have $\tau_z = 0$ if and only if $z \in -
\mathbf{On}$. Gonshor found the following sign sequence formula for monomials:

\begin{proposition}
  \label{prop-w-sign-sequence}{\tmem{{\cite[Theorem~5.11]{Gon86}}}} Let $z$ be
  a number. The sign sequence formula for $\dot{\omega}^z$ is
  \[ \dot{\omega}^z = 1 \dotplus \dot{\sum}_{\beta < \ell (x)} (z [\beta] 
     \dot{\omega}^{\tau_{z \upharpoonleft (\beta + 1)} + 1}) . \]
\end{proposition}

For instance, we have
\[ \sqrt{\omega} = \dot{\omega}^{1 / 2} \; = \; \omega \dotplus (- \omega^2)
   \text{{\hspace{3em}}and{\hspace{3em}}$\log \omega = \dot{\omega}^{\omega^{-
   1}} \; = \; \omega \dotplus (- \omega^3) .$} \]
Let us also specify a consequence that we will use often in what follows.

\begin{lemma}
  \label{lem-init-padding}Let $x, y \in \mathbf{No}$ be such that the maximal
  ordinal $\alpha$ with $\alpha \sqsubseteq y$ is a limit. We
  have~$\dot{\omega}^{x \dotplus y} = \dot{\omega}^x \dotplus
  \dot{\omega}^{\tau_x \dotplus y}$.
\end{lemma}

\begin{proof}
  Let $\alpha$ be maximal with $\alpha \sqsubseteq y$ and write $y = \alpha
  \dotplus z$. Since $y$ is strictly positive, we have~$\alpha > 0$, whence
  $\alpha \geqslant \omega$. Define
  \[ \mathfrak{m} \assign \left( z [0]  \dot{\omega}^{\tau_{\tau_x \dotplus
     \alpha \dotplus z \upharpoonleft 1} + 1} \right) \dotplus \cdots \dotplus
     \left( z [\gamma]  \dot{\omega}^{\tau_{\tau_x \dotplus \alpha \dotplus (z
     \upharpoonleft (\gamma + 1))} + 1} \right) \dotplus \cdots \]
  where $\gamma$ ranges in $\ell (z)$. Proposition~\ref{prop-w-sign-sequence}
  yields $\dot{\omega}^{\tau_x \dotplus y} = \omega^{\tau_x \dotplus \alpha}
  \dotplus \mathfrak{m}$. We also have $\omega^{x \dotplus y} = \omega^x
  \dotplus \mathfrak{n} \dotplus \mathfrak{m}$ where
  \begin{eqnarray*}
    \mathfrak{n} & = & (y [0]  \dot{\omega}^{\tau_{x \dotplus 1} + 1})
    \dotplus \cdots \dotplus (y [\beta] \omega^{\tau_{x \dotplus \beta} + 1})
    \dotplus \cdots, \beta < \alpha .
  \end{eqnarray*}
  For $\beta < \alpha$, we have $y [\beta] = 1$ and $\tau_{x \dotplus \beta} =
  \tau_x \dotplus \beta$, so
  \[ \mathfrak{n}= \dot{\omega}^{\tau_x \dotplus 1 + 1} \dotplus \cdots
     \dotplus \dot{\omega}^{\tau_x \dotplus \beta + 1} \dotplus \cdots =
     \dot{\omega}^{\tau_x \dotplus \alpha} . \]
  It follows that $\dot{\omega}^{x \dotplus y} = \dot{\omega}^x \dotplus
  \dot{\omega}^{\tau_x \dotplus y}$.
\end{proof}

\section{Surreal substructures}\label{section-surreal-substructures}

\subsection{Surreal substructures}

\begin{definition}
  A {\tmstrong{{\tmem{surreal substructure}}}} is a subclass $\mathbf{S}$ of
  $\mathbf{No}$ such that $(\mathbf{No}, <, \sqsubset)$ and $(\mathbf{S}, <,
  \sqsubset)$ are isomorphic. The isomorphism is unique, denoted
  $\Xi_{\mathbf{S}}$\label{autolab17} and called the
  {\tmstrong{{\tmem{parametrisation}}}} of $\mathbf{S}$.
\end{definition}

\begin{example}
  Here are examples of common surreal substructures and corresponding
  parametrisations.
  \begin{itemizedot}
    \item For $x \in \mathbf{No}$, the class $x \dotplus \mathbf{No} =
    \mathbf{No}^{\sqsupseteq x} \assign \{ y \in \mathbf{No} \suchthat x
    \sqsubseteq y \}$ has parametrisation $z \longmapsto x \dotplus z$. We
    have
    \begin{eqnarray*}
      \mathbf{No}^{\sqsupseteq 1} & = & \mathbf{No}^{>} \assign \{ y \in
      \mathbf{No} \suchthat y > 0 \},\\
      \mathbf{No}^{\sqsupseteq \omega} & = & \mathbf{No}^{>, \succ} \assign \{
      y \in \mathbf{No} \suchthat y >\mathbb{R} \},\\
      \mathbf{No}^{\sqsupseteq \omega^{- 1}} & = & \mathbf{No}^{>, \prec}
      \assign \{ y \in \mathbf{No} \suchthat 0 < y <\mathbb{R}^{>} \} .
    \end{eqnarray*}
    \item For $x \in \mathbf{No}^{>}$, the class $x \dottimes \mathbf{No}
    \assign \{ x \dottimes z \suchthat z \in \mathbf{No} \}$ has
    parametrisation $z \longmapsto x \dottimes z$. We have
    \begin{eqnarray*}
      1 \dottimes \mathbf{No} & = & \mathbf{No}\\
      \omega \dottimes \mathbf{No} & = & \mathbf{No}_{\succ} \assign \{ y \in
      \mathbf{No} \suchthat \tmop{supp} y \succ 1 \} .
    \end{eqnarray*}
    \item The class $\mathbf{Mo}$ of monomials has parametrisation $z \mapsto
    \dot{\omega}^z$.
  \end{itemizedot}
\end{example}

\subsection{Convexity}

We fix a surreal substructure $\mathbf{S}$. Asubclass $\mathbf{C} \subseteq
\mathbf{S}$ is said {\tmem{convex}} if for all $x, y, z \in \mathbf{S}$ with
$x \leqslant y \leqslant z$, we have $\mathord{x, y \in \mathbf{C}}
\Longrightarrow \mathord{y \in \mathbf{C}}$.

\begin{proposition}
  {\tmem{{\cite[Proposition~4.29]{BvdH19}}}} A convex subclass $\mathbf{C}$ of
  $\mathbf{S}$ is a surreal substructure if and only if each subset of
  $\mathbf{C}$ has strict upper and lower bounds in $\mathbf{C}$.
\end{proposition}

In particular, non-empty open intervals $(a, b) \cap \mathbf{S}$, for $a, b
\in \mathbf{S} \cup \{ - \infty, + \infty \}$, are surreal substructures. The
following proposition gives a sign sequence formula for $\Xi_{(a, b)}$ in
certain particular cases.

\begin{lemma}
  \label{lem-interval-iso}Let $\alpha$ be a non-zero ordinal. Let $a, b$ and
  $z$ be numbers.
  \begin{enumeratealpha}
    \item \label{lem-interval-iso-a}We have $\Xi_{(a, a \dotplus \alpha)} z =
    a \dotplus 1 \dotplus z$ if $z \ngtr \alpha_L$,
    
    and $\Xi_{(a, a \dotplus \alpha)} (\alpha \dotplus z) = a \dotplus \alpha
    \dotplus (- 1) \dotplus z$.
    
    \item \label{lem-interval-iso-b}We have $\Xi_{(b \dotplus (- \alpha), b)}
    z = b \dotplus (- 1) \dotplus z$ if $z \nless - (\alpha_L)$,
    
    and $\Xi_{(b \dotplus (- \alpha), b)} ((- \alpha) \dotplus z) = b \dotplus
    (- \alpha) \dotplus 1 \dotplus z$.
  \end{enumeratealpha}
\end{lemma}

\begin{proof}
  First note that $\alpha \sqsubseteq z$ if and only if $z > (\alpha_L)$ and
  $(- \alpha) \sqsubseteq z$ if and only if $z < - (\alpha_L)$ so the given
  descriptions cover all cases. We will only derive the formulas for the
  interval $(a, a \dotplus \alpha)$, the other ones being symmetric. The
  number $a \dotplus 1$ is the simplest element of $(a, a \dotplus \alpha)$ so
  for $z \in \mathbf{No}$, we have $\Xi_{(a, a \dotplus \alpha)} z = a
  \dotplus 1 \dotplus u_z$ for a certain number $u_z$. The class
  $\mathbf{No}^{\ngtr \alpha_L}$ is convex in $\mathbf{No}$ and for $z \in
  \mathbf{No}^{\ngtr \alpha_L}$, we have $a \dotplus 1 \dotplus z \in (a, a
  \dotplus \alpha)$. The function $z \longmapsto a \dotplus 1 \dotplus z$
  defined on $\mathbf{No}$ defines a surreal isomorphism, so we have $\Xi_{(a,
  a \dotplus \alpha)} z = a \dotplus 1 \dotplus z$ if $z \in
  \mathbf{No}^{\ngtr \alpha_L}$. On the other hand, we have
  \begin{eqnarray*}
    \mathbf{No} \setminus \mathbf{No}^{\ngtr \alpha_L} & = &
    \mathbf{No}^{\sqsupseteq \alpha}, \quad \tmop{and}\\
    (a, a \dotplus \alpha) \setminus (a \dotplus 1 \dotplus \mathbf{No}^{\ngtr
    \alpha_L}) & = & a \dotplus \alpha \dotplus (- 1) \dotplus \mathbf{No} .
  \end{eqnarray*}
  It follows that $\Xi_{(a, a \dotplus \alpha)} \upharpoonleft
  \mathbf{No}^{\sqsupseteq \alpha} = \Xi_{\mathbf{No}^{\sqsupseteq a \dotplus
  \alpha \dotplus (- 1)}}^{\mathbf{No}^{\sqsupseteq \alpha}}$, which yields
  the other part of the description.
\end{proof}

\subsection{Imbrication}

Let $\mathbf{U}$, $\mathbf{V}$ be two surreal substructures. Then there is a
unique {$(<, \sqsubset)$}-isomorphism
\[ \Xi^{\mathbf{U}}_{\mathbf{V}} \assign \Xi_{\mathbf{V}} \Xi^{-
   1}_{\mathbf{U}} : \mathbf{U} \longrightarrow \mathbf{V} \]
\label{autolab18}that we call the {\tmem{surreal isomorphism}}{\index{surreal
isomorphism}} between~$\mathbf{U}$ and $\mathbf{V}$. The composition
$\Xi_{\mathbf{U}} \circ \Xi_{\mathbf{V}}$ is also an embedding, so its image
$\mathbf{U} \mathbin{\mathbin{\Yleft}} \mathbf{V} \assign \Xi_{\mathbf{U}}
\Xi_{\mathbf{V}}  \mathbf{No}$\label{autolab19} is again a surreal
substructure that we call the {\tmem{imbrication}}{\index{imbrication of
surreal substructures}} of $\mathbf{V}$ into $\mathbf{U}$. Given $n \in
\mathbb{N}$, we write $\mathbf{U}^{\mathbin{\Yleft} n}$ for the $n$-fold
imbrication of $\mathbf{U}$ into itself. We have
$\Xi_{\mathbf{U}^{\mathbin{\Yleft} n}} = \Xi_{\mathbf{U}}^n$.

\begin{example}
  For $x, y \in \mathbf{No}$, it follows from
  Lemma~\ref{lem-surreal-ordinal-operations} that we have
  \begin{eqnarray*}
    (x \dotplus \mathbf{No}) \mathbin{\mathbin{\Yleft}} (y \dotplus
    \mathbf{No}) & = & (x \dotplus y) \dotplus \mathbf{No},\\
    (x \dottimes \mathbf{No}) \mathbin{\mathbin{\Yleft}} (y \dottimes
    \mathbf{No}) & = & (x \dottimes y) \dottimes \mathbf{No}, \text{ \ and}\\
    (x \dottimes \mathbf{No}) \mathbin{\mathbin{\Yleft}} (y \dotplus
    \mathbf{No}) & = & (x \dottimes y) \dotplus (x \dottimes \mathbf{No}) .
  \end{eqnarray*}
\end{example}

Note that given two surreal substructures $\mathbf{U}, \mathbf{V}$, we have
$\mathbf{V} \subseteq \mathbf{U}$ if and only if there is a surreal
substructure $\mathbf{W}$ with $\mathbf{U} = \mathbf{V} \mathbin{\Yleft}
\mathbf{W}$. Indeed define $\mathbf{W} = \Xi_{\mathbf{V}}^{- 1} (\mathbf{U})$.

\subsection{Fixed points}

\begin{definition}
  Let $\mathbf{S}$ be a surreal substructure. We say that a number $z$ is
  {\tmem{$\mathbf{S}$-{\tmstrong{fixed}}}}{\index{$\mathbf{S}$-fixed surreal}}
  if~$\Xi_{\mathbf{S}} z = z$. We write $\mathbf{S}^{\mathbin{\Yleft}
  \omega}$\label{autolab20} for the class of $\mathbf{S}$-fixed numbers.
\end{definition}

\begin{proposition}
  \label{prop-fixed-substructure}{\tmem{{\cite[Proposition~5.2]{BvdH19}}}} If
  $\mathbf{S}$ is a surreal substructure, then $\mathbf{S}^{\mathbin{\Yleft}
  \omega} = \bigcap_{n \in \mathbb{N}} \mathbf{S}^{\mathbin{\Yleft} n}$.
\end{proposition}

The class $\mathbf{S}^{\mathbin{\Yleft} \omega}$ may not itself be a surreal
substructure in general. In order to study cases where it is a surreal
substructure, we thus consider the following notion of closedness:

\begin{definition}
  We say that $\mathbf{S}$ is {\tmem{{\tmstrong{closed}}}} if we have
  $\Xi_{\mathbf{S}} \sup_{\sqsubseteq} Z = \sup_{\sqsubseteq} \Xi_{\mathbf{S}}
  Z$ for all non-empty set $Z \subset \mathbf{No}$ such that $(Z, \sqsubset)$
  is linearly ordered.
\end{definition}

\begin{lemma}
  \label{lem-closure-imbrication}{\tmem{{\cite[Lemma~5.12]{BvdH19}}}} If
  $\mathbf{U}$ and $\mathbf{V}$ are closed surreal substructures, then so is
  $\mathbf{U} \mathbin{\mathbin{\Yleft}} \mathbf{V}$.
\end{lemma}

\begin{proposition}
  {\tmem{{\cite[Corollary~5.14]{BvdH19}}}} If $\mathbf{S}$ is closed, then
  $\mathbf{S}^{\mathbin{\mathbin{\Yleft}} \omega}$ is a closed surreal
  substructure.
\end{proposition}

In general, when $\mathbf{S}^{\mathbin{\mathbin{\Yleft}} \omega}$ is a surreal
substructure, we write $\Xi_{\mathbf{S}^{\mathbin{\Yleft} \omega}} \backassign
\Xi_{\mathbf{S}}^{\omega}$.

\begin{example}
  \label{fixed-structures-examples}Here are some examples of structures of
  fixed points (see {\cite[Example~5.3]{BvdH19}}).
  \begin{itemizedot}
    \item For $x \in \mathbf{No}$, the structure $\mathbf{No}^{\sqsupseteq x}$
    is closed, with $\textbf{} (\mathbf{No}^{\sqsupseteq
    x})^{\mathbin{\mathbin{\Yleft}} \omega} = \mathbf{No}^{\sqsupseteq x
    \dottimes \omega}$.
    
    \item For $x \in \mathbf{No}^{>}$, the structure $x \dottimes \mathbf{No}$
    is closed, with $\textbf{} (x \dottimes
    \mathbf{No})^{\mathbin{\mathbin{\Yleft}} \omega} = x^{\infty} \dottimes
    \mathbf{No}$ where $x^{\infty} \assign \sup_{\sqsubseteq} \{ x, x
    \dottimes x, x \dottimes x \dottimes x, \ldots \}$.
    
    \item The structure $\mathbf{Mo}$ is closed, and the parametrisation
    $\Xi_{\mathbf{Mo}}^{\omega}$ of $(\mathbf{Mo})^{\mathbin{\mathbin{\Yleft}}
    \omega}$ is usually denoted $z \mapsto \varepsilon_z$\label{autolab21},
    since it extends the ordinal function $\zeta \mapsto \varepsilon_{\zeta}$
    which parametrises the fixed points of $\zeta \mapsto
    \dot{\omega}^{\zeta}$.
  \end{itemizedot}
\end{example}

\subsection{Sign sequence formulas for closed structures}

{\noindent}Let $\mathbf{S}$ be a closed surreal substructure and let
$\Psi_{\mathbf{S}} = (\alpha_z, f_z)_{z \in \textbf{} \mathbf{No}}$ denote its
alternating sign sequence formula. This formula is ``continuous'' or
``closed'' in the sense that for any limit number $z$, we have $\alpha_z =
\sup_{\sqsubseteq} \alpha_{z_{\sqsubset}}$ and $f_z = \bigcup
f_{z_{\sqsubset}}$. If for each surreal number $u$ and for $\sigma \in \{ - 1,
1 \}$ we write $\phi_{u \dotplus \sigma}$ for the non-zero surreal number
defined by $(\Xi_{\mathbf{S}} u) \dotplus \phi_{u \dotplus
\sigma}^{\mathbf{S}} = \Xi_{\mathbf{S}} (u \dotplus \sigma)$, we then have
\[ \forall z \in \textbf{} \mathbf{No}, \Xi_{\mathbf{S}} z = \Xi_{\mathbf{S}}
   0 \dotplus \phi_{z \upharpoonleft 1}^{\mathbf{S}} \dotplus \phi_{z
   \upharpoonleft 2}^{\mathbf{S}} \dotplus \cdots \dotplus \phi_{z
   \upharpoonleft (\beta + 1)}^{\mathbf{S}} \dotplus \cdots \]
where $\beta$ ranges in $\ell (z)$. Thus it suffices to compute the numbers
$\phi_{u \dotplus \sigma}$ in order to determine sign sequences of element of
$\mathbf{S}$. In the case of Gonshor's formulas for the $\omega$-map and the
$\varepsilon$-map, and Kuhlmann-Matusinski's formula for the $\kappa$-map,
those numbers are ordinals or opposites of ordinals. If $\mathbf{S}$ is not
closed, then for every limit number $a$ for which there will be an additional
term $\iota (a) \in \textbf{} \mathbf{No}$ such that $\delta (a)$ is not
always zero, and that we have, for all $z \in \mathbf{No}$, the number
$\Xi_{\mathbf{S}} z$ is
\[ \Xi_{\mathbf{S}} z = \Xi_{\mathbf{S}} 0 \dotplus \phi_{z \upharpoonleft
   1}^{\mathbf{S}} \dotplus \cdots \dotplus \delta (z \upharpoonleft \omega)
   \dotplus \phi_{z \upharpoonleft (\omega + 1)}^{\mathbf{S}} \dotplus \cdots
   \dotplus \delta (z \upharpoonleft (\omega \dottimes \eta)) \dotplus \cdots,
\]
where $\omega \dottimes \eta \leqslant \beta < \ell (z)$. Such terms $\delta
(a)$ must then be computed independently. Thus it may be a good first step
towards computing a sign sequence formula for $\mathbf{S}$ to check whether it
is closed or not, and where closure defects occur.

\begin{example}
  The parameters $\phi_{z \dotplus \sigma}$ for the $\omega$-map are
  \begin{eqnarray*}
    \phi^{\mathbf{Mo}}_{z \dotplus (- 1)} & = & \dot{\omega}^{\tau_z + 1}
    \text{\quad and}\\
    \phi^{\mathbf{Mo}}_{z \dotplus 1} & = & \dot{\omega}^{\tau_z + 2}
    \text{\quad for all $z \in \mathbf{No}$.}
  \end{eqnarray*}
  By {\cite[Chapter~9]{Gon86}}, the sign sequence formula for
  $\varepsilon_{\cdot}$ is given by
  \begin{eqnarray*}
    \phi_{z \dotplus (- 1)}^{\mathbf{Mo}^{\mathrel{\Yleft} \omega}} & = & -
    \dot{\varepsilon}_{\tau_z}^{\omega} \text{\quad and}\\
    \phi_{z \dotplus 1}^{\mathbf{Mo}^{\mathrel{\Yleft} \omega}} & = &
    \varepsilon_{\tau_{z + 1}} \text{\quad for all $z \in \mathbf{No}$.}
  \end{eqnarray*}
\end{example}

Assume again that $\mathbf{S}$ is a closed surreal substructure and let $z \in
\mathbf{No}$. Our goal in this paragraph is to compute
$\Xi_{\mathbf{S}}^{\omega} z$ provided we know $\Xi_{\mathbf{S}}^{\omega}$
on~$z_{\sqsubset}$. The identity $\mathbf{S}^{\mathrel{\Yleft} \omega} =
\bigcap_{n \in \mathbb{N}} \Xi^n_{\mathbf{S}}  \mathbf{No}$ suggests that
$\Xi_{\mathbf{S}}^{\omega}$ may be a form of limit of $\Xi^n_{\mathbf{S}}$ as
$n$ tends to infinity. Indeed, we claim that one can always find $a_z \in
\textbf{} \mathbf{No}$ such that $\Xi^{\mathbb{N}}_{\mathbf{S}} a_z$ is a
$\sqsubset$-chain with supremum $\Xi_{\mathbf{S}}^{\omega} z$. More precisely,
we have:

\begin{proposition}
  \label{prop-sign-sequence-closed}Let $\mathbf{S}$ is a closed surreal
  substructure and let $z$ be a number. We have
  \begin{eqnarray*}
    \Xi_{\mathbf{S}}^{\omega} z & = & \sup_{\sqsubset}
    \Xi_{\mathbf{S}}^{\omega} z_{\sqsubset} \text{\quad if $z$ is a limit,
    and}\\
    \Xi_{\mathbf{S}}^{\omega} z & = & \sup_{\sqsubset}
    \Xi^{\mathbb{N}}_{\mathbf{S}}  (\Xi_{\mathbf{S}}^{\omega} u \dotplus
    \sigma) \text{\quad of $z = u \dotplus \sigma$ for certain $u \in
    \mathbf{No}, \sigma \in \{ - 1, 1 \}$} .
  \end{eqnarray*}
\end{proposition}

\begin{proof}
  Let $Z = \Xi_{\mathbf{S}}^{\omega} z_{\sqsubset}$, and set $s_z =
  \sup_{\sqsubset} Z$. Since $\mathbf{S}$ is closed, so is
  $\mathbf{S}^{\mathrel{\Yleft} \omega}$, which implies that $s_z$ is
  $\mathbf{S}$-fixed.
  
  Suppose that $z$ is a limit number. Note that $s_z$ is the simplest element
  of $\mathbf{S}^{\mathrel{\Yleft} \omega}$ with $\Xi_{\mathbf{S}}^{\omega}
  z_L < s_z < \Xi_{\mathbf{S}}^{\omega} z_R$, whereas $z$ is the simplest
  number with $z_L < z < z_R$. We deduce that $s_z = \Xi_{\mathbf{S}}^{\omega}
  z$.
  
  If $z$ is not a limit, then there are a number $u$ and a sign $\sigma \in
  \{ - 1, 1 \}$ with $z = u \dotplus \sigma$. Let
  \[ a_z \assign (\Xi_{\mathbf{S}}^{\omega} u) \dotplus \sigma . \]
  The sign of $a_z - (\Xi_{\mathbf{S}}^{\omega} u)$ is $\sigma$, so the sign
  of $\Xi_{\mathbf{S}} a_z - \Xi_{\mathbf{S}}  (\Xi_{\mathbf{S}}^{\omega} u) =
  \Xi_{\mathbf{S}} a_z - (\Xi_{\mathbf{S}}^{\omega} u)$ is $\sigma$ as well,
  so $a_z \sqsubseteq \Xi_{\mathbf{S}} a_z$. It follows by induction that
  $\Xi^{\mathbb{N}}_{\mathbf{S}} a_z$ is a $\sqsubset$-chain whose supremum
  $y$ lies in each $\Xi_{\mathbf{S}}^n  \mathbf{S}$ for $n \in \mathbb{N}$ by
  closure of those structures, so $y \in \mathbf{S}^{\mathrel{\Yleft}
  \omega}$. Since the sign of $\Xi_{\mathbf{S}}^{\omega} z -
  \Xi_{\mathbf{S}}^{\omega} u$ is $\sigma$, we have $a_z \sqsubseteq
  \Xi_{\mathbf{S}}^{\omega} z$, whence $y \sqsubseteq
  \Xi_{\mathbf{S}}^{\omega} z$. Now $\Xi_{\mathbf{S}}^{\omega} z$ is the
  simplest $\mathbf{S}$-fixed number such that the sign of
  $\Xi_{\mathbf{S}}^{\omega} z - \Xi_{\mathbf{S}}^{\omega} u$ is $\sigma$,
  whence $y = \Xi_{\mathbf{S}}^{\omega} z$.
\end{proof}

\section{Exponentiation}\label{section-exponentiation}

In this section, we introduce the exponential function of
{\cite[Chapter~10]{Gon86}} as well as the surreal substructures $\mathbf{La}$
and $\mathbf{K}$.

\subsection{Surreal exponentiation}

{\tmstrong{Inductive equations}}\quad Recall that for sets $L, R$ of numbers
with $L < R$, there is a unique $\sqsubset$-minimal number $\{
L|R \}$ with
\[ L < \{ L|R \} < R. \]
Let $\mathbf{S}$ be a surreal substructure. Given subsets $L, R \subset
\mathbf{S}$ with $L < R$, there is a unique $\sqsubset$-minimal element $\{
L|R \}_{\mathbf{S}}$ of $\mathbf{S}$ with
\[ L < \{ L|R \}_{\mathbf{S}} < R. \]
In order to write certain equations, we will write, for $x \in \mathbf{S}$,
$x_L^{\mathbf{S}} =\mathbf{S} \cap x_L$ and $x_R^{\mathbf{S}} \assign x_R \cap
\mathbf{S}$. Note that $x = \{ x_L^{\mathbf{S}} |x_R^{\mathbf{S}}
\}_{\mathbf{S}}$ by definition. If $\mathbf{T}$ is another surreal
substructure, the surreal isomorphism $\Xi_{\mathbf{T}}^{\mathbf{S}}$ sends
$x$ to
\begin{equation}
  \Xi_{\mathbf{T}}^{\mathbf{S}} x = \{ \Xi_{\mathbf{T}}^{\mathbf{S}}
  x_L^{\mathbf{S}} | \Xi_{\mathbf{T}}^{\mathbf{S}} x_R^{\mathbf{S}}
  \}_{\mathbf{T}} . \label{eq-surreal-iso}
\end{equation}

{\noindent}{\tmstrong{Surreal exponentiation}} Gonshor uses an inductive
equation to define the exponential function. Given $n \in \mathbb{N}$ and $x
\in \mathbf{No}$, set $[x]_n \assign \sum_{k \leqslant n} \frac{x^k}{k!}$.
Then $\exp x$ is the number
\[ \left\{ 0, \hspace{1.2em} [x - x_L]_{\mathbb{N}} \exp x_L, \hspace{1.2em}
   [x - x_R]_{2\mathbb{N}+ 1} \exp x_R \hspace{1.2em} | \hspace{1.2em}
   \frac{\exp x_R}{[x - x_R]_{2\mathbb{N}+ 1}}, \hspace{1.2em} \frac{\exp
   x_L}{[x_L - x]_{\mathbb{N}}} \right\} . \label{autolab22} \]
This definition is warranted and $\exp$ is a strictly increasing bijective
morphism $(\mathbf{No}, +) \longrightarrow (\mathbf{No}^{>}, \cdot)$. We write
$\log$ for its functional inverse.

{\noindent}{\tmstrong{The functions $h$ and $g$}}\quad The function $\exp$
interacts with the $\omega$-map in the following way
{\cite[Chapter~10]{Gon86}}:
\[ \exp (\mathbf{Mo}^{\succ}) = \mathbf{Mo} \mathrel{\Yleft} \mathbf{Mo} . \]
More precisely, for each $x \in \mathbf{No}^{>}$ and $y \in \mathbf{No}$, we
have
\begin{eqnarray*}
  \exp (\dot{\omega}^x) & = & \dot{\omega}^{\dot{\omega}^{g (x)}},\\
  \log (\dot{\omega}^{\dot{\omega}^y}) & = & \dot{\omega}^{h (y)},
\end{eqnarray*}
where the strictly increasing and bijective function $g \of \textbf{}
\mathbf{No}^{>} \longrightarrow \textbf{} \mathbf{No}$ and its functional
inverse $h$ have the following equations in $\textbf{} \mathbf{No}^{>}$ and
$\textbf{} \mathbf{No}$ {\cite[Theorems~10.11 and~10.12]{Gon86}}:
\begin{eqnarray*}
  \forall z = \{ L|R \} \in \textbf{} \mathbf{No}^{>}, g (z) & = & \{ \Xi^{-
  1}_{\mathbf{Mo}} \mathfrak{d}_z, g (L \cap \mathbf{No}^{>}) |g (R) \} .\\
  \forall z = \{ L|R \} \in \textbf{} \mathbf{No}, h (z) & = & \{ h (L) |h
  (R), \mathbb{R}^{>}  \dot{\omega}^z \}_{\textbf{} \mathbf{No}^{>}} .
\end{eqnarray*}
The function $g$ was entirely studied by Gonshor who gave formal results such
as the characterization of its fixed points and as well as a somewhat informal
description of the sign sequence of $g (z)$ for any strictly positive number
$z$ given that of $z$. We will recover part of his results {\tmem{via}} a
different approach in Section~\ref{section-log-atomic-fixed-points}.

\subsection{Log-atomic numbers}

The class $\mathbf{La}$ of log-atomic numbers is defined as the class of
numbers $\mathfrak{m} \in \mathbf{Mo}^{\succ}$ with $\log_n \mathfrak{m} \in
\mathbf{Mo}^{\succ}$ for all $n \in \mathbb{N}$. In other words, we have\label{autolab23} $\mathbf{La} \assign \bigcap_{n \in \mathbb{N}} \exp_n  \mathbf{Mo}^{\succ}$.
We next
describe the parametrisation {\cite{BM18}} of $\mathbf{La}$. Consider, for $r \in
\mathbb{R}^{>}$ and $n \in \mathbb{N}$ the function
\[ f_{n, r} \of \mathbf{No}^{>, \succ} \longrightarrow \mathbf{No}^{>, \succ}
   \: ; \: x \mapsto \exp_n (r \log_n z) . \]
Recall that monomials are numbers $\mathfrak{m}$ which are
simplest in each class
\[ \mathcal{H} [\mathfrak{m}] = \{ x \in \mathbf{No}^{>} \suchthat \exists r
   \in \mathbb{R}^{>}, f_{0, r^{- 1}} (\mathfrak{m}) < x < f_{0, r}
   (\mathfrak{m}) \} . \]
Similarly, $\log$-atomic numbers are
the simplest numbers $\lambda$ of each class
\[ \mathcal{E} [\lambda] \assign \{ x \in \mathbf{No}^{>} \suchthat \exists n
   \in \mathbb{N}, \exists r \in \mathbb{R}^{>}, f_{n, r^{- 1}} (\mathfrak{m})
   < x < f_{n, r} (\mathfrak{m}) \} . \]
the parametrisation $\lambda_{\cdot} \assign
\Xi_{\mathbf{La}}$\label{autolab24} of $\mathbf{La}$ has {\cite[Definition~5.12 and Corollary~5.17]{BM18}} the
following equation:
\begin{eqnarray}
  \forall z \in \mathbf{No}, \lambda_z & = & \{ \mathbb{R}, f_{\mathbb{N},
  \mathbb{R}^{>}} (\lambda_{z_L}) |f_{\mathbb{N}, \mathbb{R}^{>}}
  (\lambda_{z_R}) \}  \label{eq-lambda}\\
  & = & \{ \mathbb{R}, \exp_n (r \log_n \lambda_{z'}) | \exp_n (r \log_n
  \lambda_{z''}) \}, \nonumber
\end{eqnarray}
where $r, n, z', z''$ respectively range in $\mathbb{R}^{>}$, $\mathbb{N}$,
$z_L$ and $z_R$.

Finally, we have {\cite[Proposition~2.5]{vdH:bm}}:
\begin{equation}
  \forall z \in \mathbf{No}, \lambda_{z + 1} = \exp \lambda_z .
  \label{lambda-formula}
\end{equation}

\subsection{Surreal hyperexponentiation}

In {\cite{BvdHM:surhyp}}, we defined a strictly increasing bijection $E \of
\mathbf{No}^{>, \succ} \longrightarrow \mathbf{No}^{>, \succ}$ which satisfies
the equation
\begin{equation}
  \forall x \in \mathbf{No}^{>, \succ}, E (x + 1) = \exp (E (x)) .
  \label{eq-surAbel}
\end{equation}
Morever, this function is surreal-analytic at every point in the sense of
{\cite[Definition~7.8]{BM19}} and satisfies $E (x) > \exp_{\mathbb{N}} (x)$
for all $x \in \mathbf{No}^{>, \succ}$. This function can be seen as a surreal
counterpart to Kneser's transexponential function {\cite{Kn49}}.

In order to define $E$, we relied on a surreal substructure $\mathbf{Tr}$ of
so-called truncated numbers. They can be characterized as numbers $\varphi \in
\mathbf{No}^{>, \succ}$ with $\tmop{supp} \varphi \succ
\frac{1}{\log_{\mathbb{N}} E (\varphi)}$. The function $E$ is defined at
$\varphi \in \mathbf{Tr}$ by:
\[ E (\varphi) = \{ \exp_{\mathbb{N}} (\varphi), f_{\mathbb{N},
   \mathbb{R}^{>}} (E (\varphi_L^{\mathbf{Tr}})) |f_{\mathbb{N},
   \mathbb{R}^{>}} (E (\varphi_R^{\mathbf{Tr}})) \} \in \mathbf{La} \]
This induces a strictly increasing bijection $\mathbf{Tr} \longrightarrow
\mathbf{La}$.

Now consider the set $\mathbf{No} (\varepsilon_0)$ of surreal numbers $x$ with
length $\ell (x) < \varepsilon_0$. By {\cite[Corollary~5.5]{vdDE01}}, the
structure $(\mathbf{No} (\varepsilon_0), +, \cdot, \exp)$ is an elementary
extension of the real exponential field. Moreover, for each $x \in \mathbf{No}
(\varepsilon_0)$, there is an $n \in \mathbb{N}$ with $x < \exp_n (\omega)$.
It follows that each element of $\mathbf{No} (\varepsilon_0)^{>, \succ}
\assign \mathbf{No} (\varepsilon_0) \cap \mathbf{No}^{>, \succ}$ is truncated,
and that for $\varphi \in \mathbf{No} (\varepsilon_0) \cap \mathbf{No}^{>,
\succ}$, the sets $\exp_{\mathbb{N}} (\varphi)$ and $\lambda_{\mathbb{N}} = \{
\omega, \exp (\omega), \exp (\exp (\omega)), \ldots \}$ are mutually cofinal
with respect to one another. Thus, on $\mathbf{No} (\varepsilon_0)^{>,
\succ}$, the equation for $E$ becomes
\[ \forall \varphi \in \mathbf{No} (\varepsilon_0)^{>, \succ}, E (\varphi) =
   \{ \lambda_{\mathbb{N}}, f_{\mathbb{N}, \mathbb{R}^{>}} (E
   (\varphi_L^{\mathbf{Tr}})) |f_{\mathbb{N}, \mathbb{R}^{>}} (E
   (\varphi_R^{\mathbf{Tr}})) \} . \]
\begin{proposition}
  \label{prop-E-lambda}The functions $E$ and $\lambda_{\cdot}$ coincide on
  $\mathbf{No} (\varepsilon_0)^{>, \succ}$.
\end{proposition}

We prove this by induction on $(\mathbf{No} (\varepsilon_0)^{>, \succ},
\sqsubset)$. Let $\varphi \in \mathbf{No} (\varepsilon_0)^{>, \succ}$ such
that the result holds on $\varphi_{\sqsubset}$. Since $\mathbf{No}
(\varepsilon_0)$ is $\sqsubset$-initial in $\mathbf{No}$, we get
\begin{eqnarray*}
  E (\varphi) & = & \{ \lambda_{\mathbb{N}}, f_{\mathbb{N}, \mathbb{R}^{>}} (E
  (\varphi_L^{\mathbf{Tr}})) |f_{\mathbb{N}, \mathbb{R}^{>}} (E
  (\varphi_R^{\mathbf{Tr}})) \}\\
  & = & \{ \lambda_{\mathbb{N}}, f_{\mathbb{N}, \mathbb{R}^{>}}
  (\lambda_{\varphi_L^{\mathbf{Tr}}}) |f_{\mathbb{N}, \mathbb{R}^{>}}
  (\lambda_{\varphi_R^{\mathbf{Tr}}}) \} \text{\quad (induction hypothesis)}\\
  & = & \left\{ \lambda_{\mathbb{N}}, f_{\mathbb{N}, \mathbb{R}^{>}} \left(
  \lambda_{\varphi_L^{\mathbf{No}^{>, \succ}}} \right) |f_{\mathbb{N},
  \mathbb{R}^{>}} \left( \lambda_{\varphi_R^{\mathbf{No}^{>, \succ}}} \right)
  \right\} \text{\quad ($\mathbf{No} (\varepsilon_0)^{>, \succ} \subseteq
  \mathbf{Tr}$)}\\
  & = & \{ \mathbb{R}, f_{\mathbb{N}, \mathbb{R}^{>}} (\lambda_{\varphi_L})
  |f_{\mathbb{N}, \mathbb{R}^{>}} (\lambda_{\varphi_R}) \} \text{\quad
  ($\varphi_L = \varphi_L^{>, \succ} \cup \mathbb{N}$, $\varphi_R =
  \varphi_R^{\mathbf{No}^{>, \succ}}$)}\\
  & = & \lambda_{\varphi} . \text{\quad ((\ref{eq-lambda}))}
\end{eqnarray*}
Similar relations should hold for all surreal hyperexponential functions
{\cite{vdH:hypno}} on $\mathbf{No}$.

\subsection{Kappa numbers}

The class $\mathbf{K}$\label{autolab25} of $\kappa$-numbers was introduced
by Kuhlmann and Matusinski in {\cite{KM15}}. The relation between $\mathbf{K}$ and $\mathbf{La}$ was subject to the
conjecture
$\mathbf{La} \mathrel{\overset{\mathord{?}}{=}} \log_{\mathbb{Z}}
   \mathbf{K}$ (\cite[Conjecture~5.2]{KM15}), which turned out to be false {\cite[Proposition~5.24]{BM18}}. The correct
relation between $\mathbf{La}$ and $\mathbf{K}$ was later found by Mantova and
Matusinski {\cite[p 21]{MM17}}. It is an imbrication $\mathbf{K}=
\mathbf{La} \mathbin{\mathbin{\Yleft}} \textbf{} \mathbf{No}_{\succ}$.

Similarly to monomials and log-atomic numbers, $\mathbf{K}$ numbers can be
characterized as simplest elements in each convex class
\[ \mathcal{E}^{\ast} [x] \assign \{ y \in \mathbf{No}^{>, \succ} \suchthat
   \exists n \in \mathbb{N}, \log_n x \leqslant y \leqslant \exp_n x \},
   \text{ for $x \in \mathbf{No}^{>, \succ}$} . \]
The parametrisation $\kappa_{\cdot} \assign \Xi_{\mathbf{K}} \of \mathbf{No}
\longrightarrow \mathbf{K}$\label{autolab26} of $\mathbf{K}$ has equation
\[ \forall z \in \mathbf{No}, \kappa_z = \{ \mathbb{R}, \exp_{\mathbb{N}}
   \kappa_{z_L} | \log_{\mathbb{N}} \kappa_{z_R} \} . \]
Moreover $\mathbf{K}$ is closed by {\cite[Corollary~13 and
Theorem~6.16]{BvdH19}}.

In order to compute the sign sequence of $\kappa$-numbers, Kuhlmann and
Matusinski rely on an intermediate surreal substructure denoted
$\mathbf{I}$\label{autolab27}. This surreal substructure is defined by the
imbrication relation: $\mathbf{K} \backassign \mathbf{Mo}^{\mathord{\Yleft} 2}
\mathord{\Yleft} \mathbf{I}$. Indeed, since $\mathbf{K} \subseteq \mathbf{La}
\subseteq \exp (\mathbf{Mo}^{\succ}) = \mathbf{Mo}^{\mathord{\Yleft} 2}$, the
structure $\mathbf{I}$ exists and is uniquely determined.

For $z \in \mathbf{No}$ we let $\flat z$ denote the number with $z = 1
\dotplus \flat z$ if $z > 0$ and $\flat z = - \infty$ if $z \leqslant 0$. We
also extend the functions $\varepsilon_.$ and $\tau$ to $- \infty$ with
$\tau_{- \infty} \assign - \infty$ and $\varepsilon_{- \infty} \assign 0$. For
$z \in \mathbf{No}$, we define $\phi^{\mathbf{I}}_{z \dotplus 1} \assign
\varepsilon_{\tau_{\flat z}}$ and $\phi^{\mathbf{I}}_{z \dotplus (- 1)}
\assign - \omega$. The parametrisation $\iota_{\cdot} \assign
\Xi_{\mathbf{I}}$\label{autolab28} of $\mathbf{I}$ has the following sign
sequence formula {\cite[Lemma~4.2]{KM15}}
\[ \forall z \in \mathbf{No}, \iota_z = \dot{\sum}_{\beta < \ell (z)}
   \phi^{\mathbf{I}}_{z \upharpoonleft (\beta + 1)} . \]
\section{Log-atomic numbers and fixed
points}\label{section-log-atomic-fixed-points}

Our goal is to compute the sign sequence of $\Xi_{\mathbf{La}} z = \lambda_z$
in terms of that of $z$, for all numbers~$z$. Recall that $\mathbf{La}
\subseteq \mathbf{Mo}^{\Yleft 2}$, so there is a surreal substructure
$\mathbf{R}$ with $\mathbf{La} = \mathbf{Mo}^{\Yleft 2} \Yleft
\mathbf{R}$\label{autolab29}. We write $\rho_z \assign \Xi_{\mathbf{R}}
z$\label{autolab30} for all $z \in \mathbf{No}$. We have $\mathbf{K}=
\mathbf{Mo}^{\mathord{\Yleft} 2} \mathbin{\Yleft} \mathbf{I} = \mathbf{La}
\mathbin{\Yleft} \mathbf{No}_{\succ}$, whence $\mathbf{I}=\mathbf{R}
\mathbin{\Yleft} \mathbf{No}_{\succ}$. The computation in {\cite{KM15}} of
sign sequences of numbers in $\mathbf{La} \mathbin{\Yleft}
(\mathbf{No}_{\succ} +\mathbb{Z}) = \exp_{\mathbb{Z}}  \left( \mathbf{La}
\mathbin{\Yleft} \mathbf{No}_{\succ} \right) = \log_{\mathbb{Z}} \mathbf{K}$
can thus be used to derive part of the result.

For $z = \omega \dottimes a + n \in \mathbf{No}_{\succ} +\mathbb{Z}$ where
$a \in \mathbf{No}$ and $n \in \mathbb{Z}$, we have $\lambda_z = \exp^n
\lambda_{\omega \dottimes a} = \exp^n \kappa_a$. By
{\cite[Theorem~4.3]{KM15}}, we have
\begin{eqnarray}
  \lambda_z & = & \dot{\omega}^{\dot{\omega}^{\iota_a \dotplus n}}  \quad
  \tmop{if} n \leqslant 0.  \label{eq-log-k}\\
  \lambda_z & = & \dot{\omega}^{\dot{\omega}^{\iota_a \dotplus
  \Xi_{\mathbf{Mo}}^n \left( \varepsilon_{\smash{\tau_{\flat a}}} + 1
  \right)}}  \quad \tmop{if} n > 0.  \label{eq-exp-k}
\end{eqnarray}
We want to extend this description to the sign sequences of numbers
$\rho_z$ for all $z \in \mathbf{No}$, relying on the known values
$\rho_{\omega \dottimes a + n}$ for all $a \in \mathbf{No}$ and $n \in
\mathbb{Z}$. We will compute $\rho_{\cdot}$ on all intervals
\[ \mathbf{I}_{a, n} \assign (\omega \dottimes a - (n + 1), \omega \dottimes a
   - n) \text{\quad where $a \in \mathbf{No}$ and $n \in \mathbb{Z}$.} \]
Since $\mathbf{No} = \left( \bigsqcup_{\underset{n \in \mathbb{Z}}{a \in
   \mathbf{No}}} \mathbf{I}_{a, n} \right) \sqcup \{ \omega \dottimes a + n
   \suchthat a \in \mathbf{No}, n \in \mathbb{N} \}$, this will cover all cases. The sign sequence of $\lambda_z$ can then be
computed using Proposition~\ref{prop-w-sign-sequence} twice. In order to
compute the sign sequence for $\mathbf{R}$, we describe the action of $h$ and
$g$ on sign sequences in certain cases.

\subsection{Computing $h$}

We fix $a \in \mathbf{No}$ and set $\theta_a \assign \omega \dottimes
a$.

\begin{lemma}
  \label{lem-purely-infinite-I-U}We have $\mathbf{I} \subseteq
  \mathbf{No}_{\succ}$, whence $\mathbf{R} \mathrel{\Yleft}
  \mathbf{No}_{\succ} \subseteq \mathbf{No}_{\succ}$.
\end{lemma}

\begin{proof}
  The first result is a consequence of the identity $\mathbf{No}_{\succ} =
  \omega \dottimes \mathbf{No}$ and {\cite[Lemma~4.2]{KM15}}. The second
  immediately follows.
\end{proof}

Recall that $\kappa_a = \dot{\omega}^{\dot{\omega}^{\iota_a}}$ so
$\rho_{\theta_a} = \iota_a$. {\cite[Lemma~4.2]{KM15}} gives $\tau_{\iota_a} =
\varepsilon_{\flat \tau_a}$. Set
\begin{eqnarray*}
  \varsigma_a & \assign & \tau_{\iota_a} \in \mathbf{On}, \text{\quad and}\\
  \delta_a & \assign & \dot{\omega}^{\varsigma_a \dotplus 1} \hspace{0.8em} =
  \; \dot{\omega}^{\varsigma_a} \dottimes \omega \in \mathbf{On},
\end{eqnarray*}
so $\delta_a = \omega$ if $\tau_a = 0$ and $\delta_a = \varsigma_a \dottimes
\omega$ if $\tau_a > 0$. We will treat the cases $\tau_a = 0$ and $\tau_a > 0$
in a uniform way.

\begin{lemma}
  For $n \in \mathbb{N}$, we have $\rho_{\theta_a - n} = \rho_{\theta_a} - n =
  \rho_{\theta_a} \dotplus (- n)$
\end{lemma}

\begin{proof}
  We have $\rho_{\theta_a - n} = \log^n \rho_{\theta_a} = \log^n
  \kappa_{\iota_a} = \dot{\omega}^{\dot{\omega}^{\iota_a \dotplus (- n)}}$ by
  {\cite[Theorem~4.3(1)]{KM15}}. By Lemma~\ref{lem-purely-infinite-I-U}, we
  have $\iota_a \in \mathbf{No}_{\succ}$ so $\iota_a \dotplus (- n) = \iota_a
  - n$, so $\dot{\omega}^{\dot{\omega}^{\rho_{\theta_a - n}}} = \rho_{\theta_a
  - n} = \dot{\omega}^{\dot{\omega}^{\iota_a - n}}$. Since $\rho_{\theta_a} =
  \dot{\omega}^{\dot{\omega}^{\iota_a}}$, we have $\rho_{\theta_a - n} =
  \rho_{\theta_a} - n$. It follows that $\rho_{\theta_a} - n = \rho_{\theta_a}
  \dotplus (- n)$.
\end{proof}

We now fix $n \in \mathbb{N}$ and we set $\mathbf{J}_{a, n} \assign h
(\mathbf{I}_{a, n})$. By the previous lemma, we can write the interval
$\mathbf{I}_{a, n}$ as the surreal substructure
\begin{eqnarray*}
  \mathbf{I}_{a, n} & = & (\rho_{\theta_a - n - 1}, \rho_{\theta_a - n})\\
  & = & (\rho_{\theta_a} - n - 1, \rho_{\theta_a} - n)\\
  & = & \mathbf{No}^{\sqsupseteq \rho_{\theta_a} - n - 1 / 2} .
\end{eqnarray*}
Thus $\Xi_{\mathbf{I}_{a, n}}$ is straightforward to compute in terms of sign
sequences. We have $h (\rho_{\theta_a - n}) = \dot{\omega}^{\rho_{\theta_a} -
n - 1}$ and $h (\rho_{\theta_a - n - 1}) = \dot{\omega}^{\rho_{\theta_a} - n -
2} = \dot{\omega}^{\rho_{\theta_a} - n - 1} \dotplus (- \delta_a)$. Thus the
interval $\mathbf{J}_{a, n}$ is subject to the computation given in
Lemma~\ref{lem-interval-iso}(\ref{lem-interval-iso-b}). That is, we have:

\begin{lemma}\label{lem-J}
  For $z \in \mathbf{No}$, we have
  \begin{enumeratealpha}
    \item $\Xi_{\mathbf{J}_{a, n}} z = \dot{\omega}^{\rho_{\theta_a - n - 1}}
    \dotplus (- 1) \dotplus z$ if $z \nless - (\delta_a)_L$.
    
    \item $\Xi_{\mathbf{J}_{a, n}} ((- \delta_a) \dotplus z) =
    \dot{\omega}^{\rho_{\theta_a - n - 2}} \dotplus 1 \dotplus z$.
  \end{enumeratealpha}
\end{lemma}

Moreover, for $x \in \mathbf{I}_{a, n}$, we have
$\dot{\omega}^{\rho_{\theta_a} - n - 1} - 1 \in x_L$ and
$\dot{\omega}^{\rho_{\theta_a} - n - 1} \in x_R$. This implies that
$\dot{\omega}^{\rho_{\theta_a} - n - 2} = h (\rho_{\theta_a - n - 1}) \in h
(x_L)$ and that $\dot{\omega}^{\rho_{\theta_a} - n - 1} = h (\rho_{\theta_a -
n}) \in h (x_R)$. Since with $\mathbb{R}^{>}  \dot{\omega}^x >
\dot{\omega}^{\rho_{\theta_a} - n - 2}$, we deduce the following equation for
$h$ for all $x \in \mathbf{I}_{a, n}$:
\begin{eqnarray*}
  h (x) & = & \left\{ \dot{\omega}^{\rho_{\theta_a} - n - 2}, h
  (x^{\mathbf{I}_{a, n}}_L) |h (x^{\mathbf{I}_{a, n}}_R),
  \dot{\omega}^{\rho_{\theta_a} - n - 1} \right\}\\
  & = & \{ h (x^{\mathbf{I}_{a, n}}_L) |h (x^{\mathbf{I}_{a, n}}_R)
  \}_{\mathbf{J}_{a, n}} .
\end{eqnarray*}
We see that $h \upharpoonleft \mathbf{I}_{\alpha}$ is the surreal isomorphism
$\mathbf{I}_{\alpha} \longrightarrow \mathbf{J}_{\alpha}$, so $h
\upharpoonleft \mathbf{I}_{\alpha} = \Xi_{\mathbf{J}_{\alpha}} \circ \Xi^{-
1}_{\mathbf{I}_{\alpha}}$.

\begin{proposition}
  \label{prop-h}For $z \in \mathbf{No}$, we have
  \begin{enumeratealpha}
    \item $h (\rho_{\theta_a - n - 1} \dotplus 1 \dotplus z) =
    \dot{\omega}^{\rho_{\theta_a - n - 1}} \dotplus (- 1) \dotplus z$ if $z
    \nless - (\delta_a)_L$.
    
    \item \label{prop-h-b}$h (\rho_{\theta_a - n - 1} \dotplus 1 \dotplus (-
    \delta_a) \dotplus z) = \dot{\omega}^{\rho_{\theta_a - n - 2}} \dotplus 1
    \dotplus z$.
  \end{enumeratealpha}
\end{proposition}

\begin{proof}
  This follows from Lemma~\ref{lem-J} and the identity $h
  \upharpoonleft \mathbf{I}_{a, n} = \Xi^{\mathbf{I}_{a, n}}_{\mathbf{J}_{a,
  n}}$.
\end{proof}

\begin{lemma}
  \label{lem-h-monomial}For $z \in \mathbf{I}_{a, n}$, if the number $h (z)$
  is a monomial, then there is $u \in \mathbf{No}$ with $\delta_a \sqsubseteq
  u$ and $z = \rho_{{{\theta_a} }  - n - 1} \dotplus 1 \dotplus (- \delta_a)
  \dotplus u$.
\end{lemma}

\begin{proof}
  Assume $z = \rho_{\theta_a - n - 1} \dotplus 1 \dotplus u$ where $u \nless -
  (\delta_a)_L$. Since $h (z) = \dot{\omega}^{\rho_{\theta_a - n - 1}}
  \dotplus (- 1) \dotplus u$ is a monomial with $\tau_{\rho_{\theta_a - n -
  1}} = \varsigma_a$, by Proposition~\ref{prop-w-sign-sequence}, the sign
  sequence of $u$ must start with $(- \delta_a)$, which contradicts the
  hypothesis on $u$. We deduce that there is a number $u$ with $z =
  \rho_{\theta_a - n - 1} \dotplus 1 \dotplus (- \delta_a) \dotplus u$. Since
  $h (z) = \dot{\omega}^{\rho_{\theta_a - n - 2}} \dotplus 1 \dotplus u$ is a
  monomia, the same argument entails that the sign sequence of $u$ starts with
  at least $\delta_a$ many signs $1$, that is, we must have $\delta_a
  \sqsubseteq u$.
\end{proof}

\subsection{Computing $g$}

For $n \in \mathbb{N}$, we set
\[ \mathbf{K}_{a, n} \assign \left( \dot{\omega}^{\rho_{\theta_a + n - 1}},
   \dot{\omega}^{\rho_{\theta_a + n}} \right) \text{\quad
   and\quad$\mathbf{H}_{a, n} \assign g (\mathbf{K}_{a, n}) = (\rho_{\theta_a
   + n}, \rho_{\theta_a + n + 1}) .$} \]
Recall that if $n > 0$, then $\rho_{\theta_a + n} = \rho_{\theta_a} \dotplus
\Xi_{\mathbf{Mo}}^n (\varsigma_a + 1)$, so $\Xi_{\mathbf{K}_{a, n}}$ is
subject to the computations of
Lemma~\ref{lem-interval-iso}(\ref{lem-interval-iso-a}). We have
$\rho_{\theta_a - 1} = \rho_{\theta_a} \dotplus (- 1)$, so
$\Xi_{\mathbf{K}_{a, 0}}$ is subject to the computations of
Lemma~\ref{lem-interval-iso}(\ref{lem-interval-iso-b}). Therefore:

\begin{lemma}
  For $z \in \mathbf{No}$, and $n > 0$, we have
  \begin{enumeratealpha}
    \item $\Xi_{\mathbf{K}_{a, n}} z = \dot{\omega}^{\rho_{\theta_a + n - 1}}
    \dotplus 1 \dotplus z$ if $z \ngtr (\Xi_{\mathbf{Mo}}^{n + 1} (\varsigma_a
    + 1))_L$.
    
    \item $\Xi_{\mathbf{K}_{a, n}} (\Xi_{\mathbf{Mo}}^{n + 1} (\varsigma_a +
    1) \dotplus z) = \dot{\omega}^{\rho_{\theta_a + n}} \dotplus (- 1)
    \dotplus z$ if $z > 0$.
    
    \item $\Xi_{\mathbf{K}_{a, 0}} z = \dot{\omega}^{\rho_{\theta_a}} \dotplus
    (- 1) \dotplus z$ if $z \nless - (\delta_a)_L$.
    
    \item $\Xi_{\mathbf{K}_{a, n}} ((- \delta_a) \dotplus z) =
    \dot{\omega}^{\rho_{\theta_a - 1}} \dotplus 1 \dotplus z$ if $z < 0$.
  \end{enumeratealpha}
\end{lemma}

\begin{lemma}
  For $z \in \mathbf{No}$ and $n \in \mathbb{N}$ we have
  \begin{enumeratealpha}
    \item $\Xi_{\mathbf{H}_{a, n}} z = \rho_{\theta_a + n - 1} \dotplus 1
    \dotplus z$ if $z \ngtr (\Xi_{\mathbf{Mo}}^n (\varsigma_a + 1))_L$.
    
    \item $\Xi_{\mathbf{H}_{a, n}} \left( \dot{\omega}^{\rho_{\theta_a + n -
    1}} \dotplus \Xi_{\mathbf{Mo}}^n (\varsigma_a + 1) \dotplus z \right) =
    \rho_{\theta_a + n} \dotplus (- 1) \dotplus z$ if $z > 0$.
  \end{enumeratealpha}
\end{lemma}

\begin{lemma}
  We have $g \upharpoonleft \mathbf{K}_{a, n} = \Xi_{\mathbf{H}_{a,
  n}}^{\mathbf{K}_{a, n}}$.
\end{lemma}

\begin{proof}
  Let $z \in \left( \dot{\omega}^{\rho_{\theta_a + n - 1}},
  \dot{\omega}^{\rho_{\theta_a + n}} \right)$. We have $\mathfrak{d}_z <
  \dot{\omega}^{\rho_{\theta_a + n}}$ so $\Xi_{\mathbf{Mo}}^{- 1}
  \mathfrak{d}_z < \rho_{\theta_a + n} = g \left( \dot{\omega}^{\rho_{\theta_a
  + n - 1}} \right)$. We now compute
  \begin{eqnarray*}
    g (z) & = & g \left( \left\{ z_L^{>}, \dot{\omega}^{\rho_{\theta_a + n -
    1}} |z_R^{>}, \dot{\omega}^{\rho_{\theta_a + n}} \right\}_{>} \right)\\
    & = & \left\{ \Xi_{\mathbf{Mo}}^{- 1} \mathfrak{d}_z, g \left(
    \dot{\omega}^{\rho_{\theta_a + n - 1}} \right), g (z_L^{\mathbf{K}_{a,
    n}}) |g (z_R^{\mathbf{K}_{a, n}}), g \left( \dot{\omega}^{\rho_{\theta_a +
    n}} \right) \right\}\\
    & = & \left\{ \Xi_{\mathbf{Mo}}^{- 1} \mathfrak{d}_z, g \left(
    \dot{\omega}^{\rho_{\theta_a + n - 1}} \right), g (z_L^{\mathbf{K}_{a,
    n}}) |g (z_R^{\mathbf{K}_{a, n}}), g \left( \dot{\omega}^{\rho_{\theta_a +
    n}} \right) \right\} .\\
    & = & \left\{ g \left( \dot{\omega}^{\rho_{\theta_a + n - 1}} \right), g
    (z_L^{\mathbf{K}_{a, n}}) |g (z_R^{\mathbf{K}_{a, n}}), \rho_{\theta_a}
    \dotplus \dot{\varsigma_a}^{\omega} \right\}\\
    & = & \{ g (z_L^{\mathbf{K}_{a, n}}) |g (z_R^{\mathbf{K}_{a, n}})
    \}_{\mathbf{H}_{a, n}} .
  \end{eqnarray*}
  It follows that $g \upharpoonleft \mathbf{K}_{a, n} = \Xi^{\mathbf{K}_{a,
  n}}_{\mathbf{H}_{a, n}}$.
\end{proof}

\begin{proposition}
  \label{prop-g}For $z \in \mathbf{No}$, and $n > 0$, we have
  \begin{enumeratealpha}
    \item \label{prop-g-a}$g \left( \dot{\omega}^{\rho_{\theta_a + n - 1}}
    \dotplus 1 \dotplus z \right) = \rho_{\theta_a + n} \dotplus 1 \dotplus z$
    if $z \ngtr (\Xi_{\mathbf{Mo}}^{n + 1} (\varsigma_a + 1))_L$.
    
    \item $g \left( \dot{\omega}^{\rho_{\theta_a + n - 1}} \dotplus
    \Xi_{\mathbf{Mo}}^{n + 1} (\varsigma_a + 1) \dotplus z \right) =
    \rho_{\theta_a + n + 1} \dotplus (- 1) \dotplus z$ if $z < 0$.
    
    \item $g \left( \dot{\omega}^{\rho_{\theta_a}} \dotplus (- 1) \dotplus z
    \right) = \rho_{\theta_a} \dotplus 1 \dotplus z$ if $z \nless -
    (\delta_a)_L$.
    
    \item \label{prop-g-d}$g \left( \dot{\omega}^{\rho_{\theta_a}} \dotplus (-
    \delta_a) \dotplus z \right) = \rho_{\theta_a} \dotplus z$ if $z > 0$.
  \end{enumeratealpha}
\end{proposition}

\begin{proof}
  This follows from the three previous lemmas.
\end{proof}

\subsection{The characterization theorem}

Piecing the descriptions of Propositions~\ref{prop-h} and~\ref{prop-g}
together, we obtain a full description of the sign sequence of $h (z)$ for $z
\in \mathbf{No}$. We will only require part of those descriptions to reach our
goal of describing sign sequences in $\mathbf{R}$. For $n \in \mathbb{Z}$, we
set $\mathbf{R}_{a, n} \assign \Xi_{\mathbf{R}} \mathbf{I}_{a, n}$. We will
characterize $\mathbf{R}_{a, n}$ using fixed points of the surreal
substructure
\[ \mathbf{V}_a \assign \delta_a \dotplus (- \delta_a \dottimes \omega
   \dottimes \delta_a) \dotplus \mathbf{Mo} . \]
Note that $\mathbf{V}_a$ depends only on $\tau_a$. Both $\mathbf{Mo}$ and
$\mathbf{No}^{\sqsupseteq \delta_a \dotplus (- \delta_a \dottimes \omega
\dottimes \delta_a)}$ are closed so $\mathbf{V}_a$ is closed by
Lemma~\ref{lem-closure-imbrication}. Therefore $\mathbf{V}_a^{\mathrel{\Yleft}
\omega}$ is also closed.

\begin{theorem}
  \label{th-log-atomic-caracterization}For $n \in \mathbb{N}$, we have
  \begin{enumeratealpha}
    \item $\mathbf{R}_{a, n} = \rho_{\theta_a - (n + 1)} \dotplus 1 \dotplus
    (- \delta_a) \dotplus \mathbf{V}_a^{\mathrel{\Yleft} \omega}$.
    
    \item For $u = \rho_{\theta_a - (n + 1)} \dotplus 1 \dotplus (- \delta_a)
    \dotplus v$ where $v \in \mathbf{V}_a^{\mathrel{\Yleft} \omega}$ and $k
    \in \mathbb{N}$, we have $\log^k  \dot{\omega}^{\dot{\omega}^u} =
    \dot{\omega}^{\dot{\omega}^{\rho_{\theta_a - (n + k + 1)} \dotplus 1
    \dotplus (- \delta_a) \dotplus v}}$.
    
    \item $\mathbf{R}_{a, - (n + 1)} = \rho_{\theta_a + n} \dotplus
    \mathbf{V}_a^{\mathrel{\Yleft} \omega}$.
    
    \item For $u = \rho_{\theta_a - (n + 1)} \dotplus 1 \dotplus (- \delta_a)
    \dotplus v$ where $v \in \mathbf{V}_a^{\Yleft \omega}$ and $k \in
    \mathbb{N}^{>}$, we have
    \[ \exp^k  \dot{\omega}^{\dot{\omega}^u} =
       \dot{\omega}^{\dot{\omega}^{\rho_{\theta_a + k - (n + 1)} \dotplus v}}
       . \]
  \end{enumeratealpha}
\end{theorem}

\begin{proof}
  We first prove by induction on $k$ that we have
  \[ \forall n \in \mathbb{N}, \mathbf{R}_{a, n} \subseteq \rho_{\theta_a - (n
     + 1)} \dotplus 1 \dotplus (- \delta_a) \dotplus
     \mathbf{V}_a^{\mathrel{\Yleft} k} . \]
  First assume that $k = 0$. For $n \in \mathbb{N}$ and $u \in \mathbf{R}_{a,
  n}$, the number $h (u)$ is a monomial, which implies by
  Lemma~\ref{lem-h-monomial} that $u$ has the form $\rho_{\theta_a - (n + 1)}
  \dotplus 1 \dotplus (- \delta_a) \dotplus z$ with~$\delta_a \sqsubseteq z$.
  This implies in particular that $u \in \rho_{\theta_a - (n + 1)} \dotplus 1
  \dotplus (- \delta_a) \dotplus \mathbf{No}$.
  
  Let $k \in \mathbb{N}$ such that the result holds at $k$ and consider $n \in
  \mathbb{N}$ and $u_0 \in \mathbf{R}_{a, n}$. There is a $v_k \in
  \mathbf{V}_a^{\mathrel{\Yleft} k}$ with $u_0 = \rho_{\theta_a - (n + 1)}
  \dotplus 1 \dotplus (- \delta_a) \dotplus v_k$. Since $u_0 \in
  \mathbf{R}_{a, n}$, we have $h (u_0) \in \dot{\omega}^{\mathbf{R}_{a, n +
  1}}$. Now $h (u_0) = \dot{\omega}^{\rho_{\theta_a - (n + 2)}} \dotplus 1
  \dotplus v_k = \dot{\omega}^{\rho_{\theta_a - (n + 2)}} \dotplus v_k$ by
  Proposition~\ref{prop-h}(\ref{prop-h-b}). Let $u_1 \in \mathbf{R}_{a, n +
  1}$ with $h (u_0) = \dot{\omega}^{{u_1} }$. We have $h (u_1) \in
  \mathbf{Mo}$, so by Lemma~\ref{lem-h-monomial}, there is a number $v_{k +
  1}$ with $u_1 = \rho_{\theta_a - (n + 2)} \dotplus 1 \dotplus (- \delta_a)
  \dotplus v_{k + 1}$ and $\delta_a \sqsubseteq v_{k + 1}$. Thus $h (u_0) =
  \dot{\omega}^{\rho_{\theta_a - (n + 2)} \dotplus 1 \dotplus (- \delta_a)
  \dotplus v_{k + 1}}$. By Proposition~\ref{prop-w-sign-sequence} and
  Lemma~\ref{lem-init-padding}, we have
  \[ h (u_0) = \dot{\omega}^{\rho_{\theta_a - (n + 2)}} \dotplus \delta_a
     \dotplus (- \delta_a \dottimes \omega \dottimes \delta_a) \dotplus
     \dot{\omega}^{\varsigma_a \dotplus 1 \dotplus v_{k + 1}} . \]
  We identify
  \begin{eqnarray*}
    v_k & = & \delta_a \dotplus (- \delta_a \dottimes \omega \dottimes
    \delta_a) \dotplus \dot{\omega}^{\varsigma_a \dotplus 1 \dotplus v_{k +
    1}}\\
    & = & \delta_a \dotplus (- \delta_a \dottimes \omega \dottimes \delta_a)
    \dotplus \dot{\omega}^{v_{k + 1}}  \text{\qquad ($\delta_a = \varsigma_a
    \dotplus 1 \dotplus \delta_a \sqsubseteq v_{k + 1}$)}\\
    & = & \Xi_{\mathbf{V}_a} v_{k + 1} .
  \end{eqnarray*}
  The inductive hypothesis applied at $(u_1, n + 1)$ yields $v_{k + 1} \in
  \mathbf{V}_a^{\mathrel{\Yleft} k}$, so $v_k \in
  \mathbf{V}_a^{\mathrel{\Yleft} (k + 1)}$. We thus have
  \[ \mathbf{R}_{a, n} \subseteq \bigcap_{k \in \mathbb{N}} \rho_{\theta_a -
     (n + 1)} \dotplus 1 \dotplus (- \delta_a) \dotplus
     \mathbf{V}_a^{\mathrel{\Yleft} k} = \rho_{\theta_a - (n + 1)} \dotplus 1
     \dotplus (- \delta_a) \dotplus \mathbf{V}_a^{\mathrel{\Yleft} \omega} .
  \]
  We next prove $b$). Since it implies that $\log^{\mathbb{N}} 
  \dot{\omega}^{\dot{\omega}^{\mathbf{R}_{a, n}'}} \subseteq \mathbf{Mo}$, \
  this will yield $a$). We prove $b$) by induction on $k$. Note that the
  result is immediate for $k = 0$. Let $k \in \mathbb{N}$ be such that the
  formula holds at $k$, let $v \in \mathbf{V}_a^{\mathrel{\Yleft} \omega}$,
  and set $u = \rho_{\theta_a - (n + 1)} \dotplus 1 \dotplus (- \delta_a)
  \dotplus v$ and $\mathfrak{m} \assign \dot{\omega}^{\dot{\omega}^u}$. Our
  inductive hypothesis is
  \begin{eqnarray*}
    \log_k \mathfrak{m} & = & \dot{\omega}^{\dot{\omega}^{\rho_{\theta_a - (n
    + k + 1)} \dotplus 1 \dotplus (- \delta_a) \dotplus v}} .
  \end{eqnarray*}
  We have
  \begin{eqnarray*}
    \log_{k + 1} \mathfrak{m} & = & \log
    \dot{\omega}^{\dot{\omega}^{\rho_{\theta_a - (n + k + 1)} \dotplus 1
    \dotplus (- \delta_a) \dotplus v}}\\
    & = & \dot{\omega}^{h (\rho_{\theta_a - (n + k + 1)} \dotplus 1 \dotplus
    (- \delta_a) \dotplus v)}\\
    & = & \dot{\omega}^{\dot{\omega}^{\rho_{\theta_a - (n + k + 2)}} \dotplus
    1 \dotplus v} \text{ (Proposition~\ref{prop-h}(\ref{prop-h-b}))}\\
    & = & \dot{\omega}^{\dot{\omega}^{\rho_{\theta_a - (n + k + 2)}} \dotplus
    v}\\
    & = & \dot{\omega}^{\dot{\omega}^{\rho_{\theta_a - (n + k + 2)}} \dotplus
    \delta_a \dotplus (- \delta_a \dottimes \omega \dottimes \delta_a)
    \dotplus \dot{\omega}^v} \text{ ($v \in \mathbf{V}_a^{\mathrel{\Yleft}
    \omega} = \mathbf{Fix}_{\mathbf{V}_a}$)}\\
    & = & \dot{\omega}^{\dot{\omega}^{\rho_{\theta_a - (n + k + 2)}} \dotplus
    \delta_a \dotplus (- \delta_a \dottimes \omega \dottimes \delta_a)
    \dotplus \dot{\omega}^{\varsigma_a \dotplus 1 \dotplus v}} \text{
    ($\delta_a = \varsigma_a \dotplus 1 \dotplus \delta_a \sqsubseteq v_{k +
    1}$)}\\
    & = & \dot{\omega}^{\dot{\omega}^{\rho_{\theta_a - (n + k + 2)} \dotplus
    1 \dotplus (- \delta_a)} \dotplus \dot{\omega}^{\varsigma_a \dotplus 1
    \dotplus v}} \text{\qquad (Proposition~\ref{prop-w-sign-sequence})}\\
    & = & \dot{\omega}^{\dot{\omega}^{\rho_{\theta_a - (n + k + 2)} \dotplus
    1 \dotplus (- \delta_a) \dotplus v}} . \text{\qquad
    (Lemma~\ref{lem-init-padding})}
  \end{eqnarray*}
  The formula follows by induction.
  
  We next prove $d$). Let $u = \rho_{\theta_a - 1} \dotplus 1 \dotplus (-
  \delta_a) \dotplus v \in \mathbf{R}_{a, 0}$ where~$v \in
  \mathbf{V}_a^{\mathrel{\Yleft} \omega}$. By the same computations as above,
  we have $g (\dot{\omega}^u) = g \left( \dot{\omega}^{\rho_{\theta_a - 1}
  \dotplus 1 \dotplus (- \delta_a) \dotplus v} \right)$.
  Proposition~\ref{prop-w-sign-sequence} and Lemma~\ref{lem-init-padding} give
  $g (\dot{\omega}^u) = g \left( \dot{\omega}^{\rho_{\theta_a - 1}} \dotplus
  \delta_a \dotplus (- \delta_a \dottimes \omega \dottimes \delta_a) \dotplus
  \dot{\omega}^{\varsigma_a \dotplus 1 \dotplus v} \right)$. Now $\delta_a =
  \varsigma_a \dotplus 1 \dotplus \delta_a \sqsubseteq v$, so
  \begin{eqnarray*}
    g (\dot{\omega}^u) & = & g \left( \dot{\omega}^{\rho_{\theta_a} \dotplus
    (- 1)} \dotplus \delta_a \dotplus (- \delta_a \dottimes \omega \dottimes
    \delta_a) \dotplus \dot{\omega}^v \right)\\
    & = & g \left( \dot{\omega}^{\rho_{\theta_a} \dotplus (- 1)} \dotplus v
    \right) \text{\quad  ($v \in \mathbf{V}_a^{\mathrel{\Yleft} \omega}$)}\\
    & = & g \left( \dot{\omega}^{\rho_{\theta_a}} \dotplus (- \delta_a)
    \dotplus v \right) \text{\quad (Proposition~\ref{prop-w-sign-sequence})}\\
    & = & \rho_{\theta_a} \dotplus 1 \dotplus v \text{\quad
    (Proposition~\ref{prop-g}(\ref{prop-g-d}))}\\
    & = & \rho_{\theta_a} \dotplus v.
  \end{eqnarray*}
  This proves the formula for $n = 0$. Now let $n \in \mathbb{N}$ be such that
  the formula holds at $n$. Let $u = \rho_{\theta_a + n} \dotplus v \in
  \mathbf{R}_{a, - (n + 1)}$ where $v \in \mathbf{V}_a^{\mathrel{\Yleft}
  \omega}$. We have $v < \delta_a$ so $v \ngtr (\Xi_{\mathbf{Mo}}^{n + 1}
  (\delta_a + 1))_L$, thus Proposition~\ref{prop-g}(\ref{prop-g-a}) yields
  \[ g (\dot{\omega}^u) = g \left( \dot{\omega}^{\rho_{\theta_a + n} \dotplus
     v} \right) = g \left( \dot{\omega}^{\rho_{\theta_a + n} \dotplus 1
     \dotplus v} \right) = \rho_{\theta_a + n + 1} \dotplus v. \]
  By induction, the formula for $g$ holds for all $n \in \mathbb{N}$.
\end{proof}

\section{Applications}\label{section-applications}

\subsection{Closedness of $\mathbf{La}$}

\begin{corollary}
  \label{cor-U-partial-closure}For all $n \in \mathbb{Z}$, the structure
  $\mathbf{R}_{a, n}$ is closed.
\end{corollary}

\begin{proof}
  Both $\mathbf{No}^{\sqsupseteq \delta_a \dotplus (- \delta_a \dottimes
  \omega \dottimes \delta_a)}$ and $\mathbf{Mo}$ are closed. By
  Lemma~\ref{lem-closure-imbrication}, the class~$\mathbf{V}_a$ is also
  closed. Therefore $\mathbf{V}_a^{\mathrel{\Yleft} \omega}$ is closed. Recall
  that $\mathbf{R}_{a, n} = \mathbf{No}^{\sqsupseteq b} \mathbin{\Yleft}
  \mathbf{V}_a^{\mathrel{\Yleft} \omega}$ for a certain number~$b$. We
  conclude with Lemma~\ref{lem-closure-imbrication}.
\end{proof}

\begin{proposition}
  \label{prop-La-closed}The structure $\mathbf{La}$ is closed.
\end{proposition}

\begin{proof}
  Let $C$ be a non-empty chain in $\mathbf{La}$. We want to prove that
  $\sup_{\sqsubset} C$ is log-atomic. We may assume that $C$ has no
  $\sqsubset$-maximum.
  
  If the set $K$ of $\kappa$-numbers $\kappa_z$ for which there is a $c \in
  C$ with $\kappa_z \sqsubseteq c$ is cofinal in $(C, \sqsubset)$, then
  $\sup_{\sqsubset} C = \sup_{\sqsubset} K \in \mathbf{K}$ by closure of
  $\mathbf{K}$, so $\sup_{\sqsubset} C \in \mathbf{La}$.
  
  Assume that $K$ is not cofinal in $(C, \sqsubset)$. Let $c_0 \in C$ be
  $\sqsubset$-minimal with $K \sqsubset c_0$. Let $C_0 = \{ c \in C : c_0
  \sqsubseteq c \}$ and notice that $\sup_{\sqsubset} C = \sup_{\sqsubset}
  C_0$. Let $a \in \mathbf{No}$ with $\kappa_a = \sup_{\sqsubset} K \in
  \mathbf{K}$. For $c \in C_0$, the number $\kappa_a$ is $\sqsubset$-maximal
  in $\mathbf{K}$ with $\kappa_a \sqsubseteq c$. By {\cite[Corollary~7.13 and
  Theorem~6.16]{BvdH19}}, we have $c \in \mathcal{E}^{\ast} [\kappa_a]$ for
  all $c \in C_0$, whence $C_0 \subseteq \mathcal{E}^{\ast} [\kappa_a]
  =\mathcal{E}^{\ast} [\lambda_{\theta_a}]$. The set $\log^{\mathbb{Z}}
  \lambda_{\theta_a}$ is cofinal and coinitial in $\mathcal{E}^{\ast}
  [\lambda_{\theta_a}]$ for the order $\leqslant$. So for each element $c \in
  C_0$ there is a unique integer $n_c \in \mathbb{Z}$ with $c \in \mathbf{La}
  \mathbin{\Yleft} \mathbf{I}_{a, n_c}$. The equations~(\ref{eq-log-k}) and
  (\ref{eq-exp-k}) imply that if $c, c'$ are elements of $C_0$ with $c
  \sqsubset c'$, then we must have $n_c \sqsubseteq n_{c'}$. It follows that
  $\nu \assign \sup_{\sqsubset} n_{C_0}$ exists in $\mathbb{Z} \cup \{ -
  \omega, \omega \}$. If $\nu = \sigma \omega$ where $\sigma \in \{ - 1, 1
  \}$, then $\sup_{\sqsubset} C = \lambda_{\theta_{a \dotplus \sigma}} \in
  \mathbf{La}$. If $\nu \in \mathbb{Z}$, then for $c_1 \in C_0$ with $n_{c_1}
  = \nu$ and $C_1 = \{ c \in C : c_1 \sqsubseteq c \}$, we have
  $\sup_{\sqsubset} C = \sup_{\sqsubset} C_1$ and $C_1 \subseteq \mathbf{La}
  \mathbin{\Yleft} \mathbf{I}_{a, \nu} = \mathbf{Mo}^{\Yleft 2}
  \mathbin{\Yleft} \mathbf{R}_{a, \nu}$. The structure $\mathbf{Mo}^{\Yleft
  2}$ is closed, so by Corollary~\ref{cor-U-partial-closure} and
  Lemma~\ref{lem-closure-imbrication}, the structure $\mathbf{Mo}^{\Yleft 2}
  \mathbin{\Yleft} \mathbf{R}_{a, \nu}$ is closed. So $\sup_{\sqsubset} C \in
  \mathbf{Mo}^{\mathrel{\Yleft} 2} \mathbin{\Yleft} \mathbf{R}_{a, \nu}$ is
  log-atomic.
\end{proof}

\subsection{The sign sequence formula}\label{subsection-the-formula}

We apply the results of the previous sections to give the sign sequence
formula for $\mathbf{R}$.

\begin{proposition}
  Let $V (a) = \sup_{\sqsubset} \Xi_{\mathbf{V}_a}^{\mathbb{N}} 0$. The number
  $V (a) = \Xi_{\mathbf{V}_a}^{\omega} 0$ is the transfinite concatenation
  \[ \dot{\omega}^{\varsigma_a \dotplus 1} \dotplus (-
     \dot{\omega}^{\varsigma_a \dotplus 2 \dotplus \varsigma_a \dotplus 1}) 
     \dotplus \dot{\sum}_{n \in \mathbb{N}} \dot{\omega}^{\Xi_{\mathbf{Mo}}^n
     \delta_a} \dotplus \dot{\omega}^{\left( \sum_{k = 0}^n
     \Xi_{\mathbf{Mo}}^k \delta_a \right) \dotplus 1 \dotplus \varsigma_a
     \dotplus 2 \dotplus \varsigma_a \dotplus 1} . \]
  As a consequence, we have $\tau_{V (a)} = \varepsilon_{\flat \tau_{a
  \dotplus 1}} \in \mathbf{Mo}^{\mathrel{\Yleft} \omega}$.
\end{proposition}

\begin{proof}
  The identity $V (a) = \sup_{\sqsubset} \Xi_{\mathbf{V}_a}^{\mathbb{N}} 0$ is
  a direct consequence of the closure of $\mathbf{V}_a$. The computation of
  this supremum is left to the reader.
\end{proof}

\begin{theorem}\label{th-formula}
  For $z \in \mathbf{No}$ define $\phi^a_{z \dotplus 1} = \varepsilon_{\flat
  \tau_{a \dotplus 1 \dotplus z \dotplus 1}}$ and $\phi^a_{z \dotplus (- 1)} =
  - \dot{\varepsilon}_{\mathord{\flat}  \smash{\tau_{a \dotplus 1 \dotplus
  z}}}^{\omega}$. Let $\Phi^a$ be the function defined on $\mathbf{No}$ by
  $\Phi^a  (z) = \dot{\sum}_{\beta < \ell (z)} {\phi^a_{z \upharpoonleft
  (\beta + 1)}} $. For all $z \in \mathbf{No}$, we have
  \[ \Xi_{\mathbf{V}_a}^{\omega} z = V (a) \dotplus \Phi^a (z) . \]
\end{theorem}

\begin{proof}
  Since $\mathbf{V}_a$ is closed, we may rely on
  Proposition~\ref{prop-sign-sequence-closed}. We need only prove that for $z
  \in \mathbf{No}$ and $\sigma \in \{ - 1, 1 \}$, we have
  $\Xi_{\mathbf{V}_a}^{\omega} z \dotplus \phi^a_{z \dotplus \sigma} =
  \sup_{\sqsubseteq} \Xi_{\mathbf{V}_a}^{\mathbb{N}}
  (\Xi_{\mathbf{V}_a}^{\omega} z \dotplus \sigma)$. We prove this by induction
  on $(\mathbf{No}, \sqsubset)$ along with the claim that $\forall z \in
  \mathbf{No}, \tau_{\Xi_{\mathbf{V}_a}^{\omega} z} = \varepsilon_{\flat
  \tau_{a \dotplus 1 \dotplus z}}$.
  
  Note that the functions $\tau_.$, $\varepsilon_{\flat .}$, $z \mapsto a
  \dotplus 1 \dotplus z$ and $\Xi_{\mathbf{V}_a}^{\omega}$ preserve non-empty
  suprema. Moreover, the identity $\tau_{\Xi_{\mathbf{V}_a}^{\omega} z} =
  \varepsilon_{\flat \tau_{a \dotplus 1 \dotplus z}}$ is valid for $z = 0$. So
  by the previous lemma, we need only prove the claim at successors cases. Let
  $z \in \mathbf{No}$, set $\beta_0 = 0$, and define $\beta_{n + 1} \assign
  \dot{\omega}^{\varepsilon_{\flat \tau_{a \dotplus 1 \dotplus z}} \dotplus 1
  \dotplus \beta_n}$ and $a_n \assign \Xi_{\mathbf{V}_a}^n
  (\Xi_{\mathbf{V}_a}^{\omega} z \dotplus 1)$ for each $n \in \mathbb{N}$. We
  have $\sup_{\sqsubset} \beta_{\mathbb{N}} = \varepsilon_{\flat \tau_{a
  \dotplus 1 \dotplus z \dotplus 1}} = \phi^a_{z \dotplus 1}$. Consider an $n
  \in \mathbb{N}$ with $a_n = \Xi_{\mathbf{V}_a}^{\omega} z \dotplus 1
  \dotplus \beta_n$ (this is the case for $n = 0$). We have
  \begin{eqnarray*}
    a_{n + 1} & = & \Xi_{\mathbf{V}_a} (\Xi_{\mathbf{V}_a}^{\omega} z \dotplus
    1 \dotplus \beta_n)\\
    & = & \delta_a \dotplus (- \delta_a \dottimes \omega \dottimes \delta_a)
    \dotplus \dot{\omega}^{\Xi_{\mathbf{V}_a}^{\omega} z \dotplus 1 \dotplus
    \beta_n}\\
    & = & \delta_a \dotplus (- \delta_a \dottimes \omega \dottimes \delta_a)
    \dotplus \dot{\omega}^{\Xi_{\mathbf{V}_a}^{\omega} z} \dotplus
    \dot{\omega}^{\smash{\tau_{\Xi_{\mathbf{V}_a}^{\omega} z} \dotplus 1
    \dotplus \beta_n}}\\
    & = & \Xi_{\mathbf{V}_a} \Xi_{\mathbf{V}_a}^{\omega} z \dotplus
    \dot{\omega}^{\varepsilon_{\flat \smash{\tau_{a \dotplus 1 \dotplus z}}}
    \dotplus 1 \dotplus \beta_n}\\
    & = & \Xi_{\mathbf{V}_a}^{\omega} z \dotplus \beta_{n + 1}\\
    & = & \Xi_{\mathbf{V}_a}^{\omega} z \dotplus 1 \dotplus \beta_{n + 1} .
  \end{eqnarray*}
  It follows that $\sup_{\sqsubset} a_{\mathbb{N}} =
  \Xi_{\mathbf{V}_a}^{\omega} z \dotplus \sup_{\sqsubset} \beta_{\mathbb{N}} =
  \Xi_{\mathbf{V}_a}^{\omega} z \dotplus \phi^a_{z \dotplus 1}$. The second
  claim is valid since we have $\tau_{\Xi_{\mathbf{V}_a}^{\omega}  (z \dotplus
  1)} = \tau_{\Xi_{\mathbf{V}_a}^{\omega} z} {\dotplus \phi_{z \dotplus 1}^a} 
  = \varepsilon_{\flat \tau_{a \dotplus 1 \dotplus z}} \dotplus
  \varepsilon_{\flat \tau_{a \dotplus 1 \dotplus z \dotplus 1}} =
  \varepsilon_{\flat \tau_{a \dotplus 1 \dotplus z \dotplus 1}}$. Let
  $\gamma_0 = 1$, and for $n \in \mathbb{N}$ define $\gamma_{n + 1} =
  \varepsilon_{\flat \tau_{a \dotplus 1 \dotplus z}} \dottimes \omega
  \dottimes \gamma_n$. We have $\sup_{\sqsubset} (- \gamma_{\mathbb{N}}) = -
  \dot{\varepsilon}_{\flat \tau_{a \dotplus 1 \dotplus z}}^{\omega} =
  \phi^a_{z \dotplus (- 1)}$. Let $n \in \mathbb{N}$ with $a_n =
  \Xi_{\mathbf{V}_a}^{\omega} z \dotplus (- \gamma_n)$ (this is the case for
  $n = 0$). We have
  \begin{eqnarray*}
    a_{n + 1} & = & \Xi_{\mathbf{V}_a} (\Xi_{\mathbf{V}_a}^{\omega} z \dotplus
    (- \gamma_n))\\
    & = & \delta_a \dotplus (- \delta_a \dottimes \omega \dottimes \delta_a)
    \dotplus \dot{\omega}^{\Xi_{\mathbf{V}_a}^{\omega} z \dotplus (-
    \gamma_n)}\\
    & = & \delta_a \dotplus (- \delta_a \dottimes \omega \dottimes \delta_a)
    \dotplus \dot{\omega}^{\Xi_{\mathbf{V}_a}^{\omega} z} \dotplus \left( -
    \omega^{\smash{\tau_{\Xi_{\mathbf{V}_a}^{\omega} z} \dotplus 1}} \dottimes
    \gamma_n \right) \text{\quad ($\tmop{Proposition} \tmxspace
    \ref{prop-w-sign-sequence}$)}\\
    & = & \Xi_{\mathbf{V}_a} \Xi_{\mathbf{V}_a}^{\omega} \left( z \dotplus
    \left( - \dot{\omega}^{\smash{\varepsilon_{\flat \tau_{a \dotplus 1
    \dotplus z}}}} \dottimes \omega \dottimes \gamma_n \right) \right)\\
    & = & \Xi_{\mathbf{V}_a}^{\omega} z \dotplus (- \varepsilon_{\flat
    \tau_{a \dotplus 1 \dotplus z}} \dottimes \omega \dottimes \gamma_n)\\
    & = & \Xi_{\mathbf{V}_a}^{\omega} z \dotplus (- \gamma_{n + 1}) .
  \end{eqnarray*}
  It follows that $\sup_{\sqsubset} a_{\mathbb{N}} =
  \Xi_{\mathbf{V}_a}^{\omega} z \dotplus \phi^a_{z \dotplus (- 1)}$. The
  second claim is valid since we have $\tau_{\Xi_{\mathbf{V}_a}^{\omega}  (z
  \dotplus (- 1))} = \tau_{\Xi_{\mathbf{V}_a}^{\omega} z} = \varepsilon_{\flat
  \tau_{a \dotplus 1 \dotplus z}} = \varepsilon_{\flat \tau_{a \dotplus 1
  \dotplus z \dotplus (- 1)}}$. This concludes the proof.
\end{proof}

In order to obtain the formula for $\lambda_z$, one need only use the relation
$\lambda_z = \dot{\omega}^{\dot{\omega}^{\rho_z}}$ and apply
Proposition~\ref{prop-w-sign-sequence} twice.

\end{document}